\documentclass{amsart}
\usepackage{amsmath}

\usepackage{amssymb}
\usepackage{amsthm}
\usepackage{amscd,amsbsy}
\usepackage{xcolor}
\usepackage[T1]{fontenc}

\newtheorem{theorem}{Theorem}[section]
\newtheorem{lemma}[theorem]{Lemma}
\newtheorem{proposition}[theorem]{Proposition}
\newtheorem{corollary}[theorem]{Corollary}
\newtheorem{claim}[theorem]{Claim}

\theoremstyle{definition}
\newtheorem{definition}[theorem]{Definition}
\newtheorem{definitions}[theorem]{Definitions}

\newtheorem{definitions and remarks}[theorem]{Definitions and Remarks}

\theoremstyle{remark}
\newtheorem{notation}[theorem]{Notation}
\newtheorem{remark}[theorem]{Remark}
\newtheorem{remarks}[theorem]{Remarks}

\numberwithin{equation}{section}

\newcommand{\order}{\mathrm{order}\,}

\newcommand{\al}{{\alpha}}
\newcommand{\be}{{\beta}}
\newcommand{\de}{{\delta}}
\newcommand{\ep}{{\epsilon}}
\newcommand{\D}{{\Delta}}
\newcommand{\ga}{{\gamma}}
\newcommand{\Ga}{{\Gamma}}

\newcommand{\la}{{\lambda}}

\newcommand{\Om}{{\Omega}}
\newcommand{\p}{{\partial}}
\newcommand{\s}{{\sigma}}

\newcommand{\vp}{{\varphi}}

\newcommand{\IN}{{\mathbb N}}

\newcommand{\IR}{{\mathbb R}}
\newcommand{\IC}{{\mathbb C}}

\newcommand{\cC}{{\mathcal C}}

\newcommand{\cF}{{\mathcal F}}

\newcommand{\cI}{{\mathcal I}}
\newcommand{\cJ}{{\mathcal J}}

\newcommand{\cO}{{\mathcal O}}
\newcommand{\cQ}{{\mathcal Q}}

\newcommand{\ta}{{\tilde a}}
\newcommand{\tb}{{\tilde b}}
\newcommand{\tc}{{\tilde c}}

\newcommand{\tB}{{\widetilde B}}

\newcommand{\tG}{{\widetilde G}}

\newcommand{\tbe}{{\tilde \be}}
\newcommand{\txi}{{\tilde \xi}}
\newcommand{\teta}{{\tilde \eta}}

\newcommand{\hb}{{\hat b}}
\newcommand{\hc}{{\hat c}}
\newcommand{\hf}{{\hat f}}
\newcommand{\hg}{{\hat g}}
\newcommand{\hh}{{\hat h}}
\newcommand{\hal}{{\hat \al}}

\newcommand{\hga}{{\hat \ga}}
\newcommand{\hs}{{\hat \s}}
\newcommand{\hvp}{{\hat \vp}}

\newcommand{\wK}{{\widehat K}}

\newcommand{\llb}{{[\![}}
\newcommand{\rrb}{{]\!]}}

\newcommand{\RN}[1]{%
  \textup{\uppercase\expandafter{\romannumeral#1}}%
}

\begin{document}
\title[Solutions of quasianalytic equations]
{Solutions of quasianalytic equations}

\author[A.~Belotto]{Andr\'e Belotto da Silva}
\author[I.~Biborski]{Iwo Biborski}
\author[E.~Bierstone]{Edward Bierstone}
\address{University of Toronto, Department of Mathematics, 40 St. George Street,
Toronto, ON, Canada M5S 2E4}
\curraddr[A.~Belotto]{Universit\'e Paul Sabatier, Institut de Math\'ematiques de Toulouse,
118 route de Narbonne, F-31062 Toulouse Cedex 9, France}
\email[A.~Belotto]{andre.belotto\_da\_silva@math.univ-toulouse.fr}
\email[I.~Biborski]{iwobiborski@gmail.com}
\email[E.~Bierstone]{bierston@math.toronto.edu}
\thanks{Research supported in part by NSERC grant OGP0009070}



\begin{abstract}
The article develops techniques for solving equations $G(x,y)=0$, where $G(x,y)=G(x_1,\ldots,x_n,y)$ is
a function in a given quasianalytic class (for example, a quasianalytic Denjoy-Carleman class, or the class
of $\cC^\infty$ functions definable in a polynomially-bounded $o$-minimal structure). We show that, if $G(x,y)=0$ 
has a formal power series solution $y=H(x)$ at some point $a$, then $H$ is the Taylor expansion at $a$ of a
quasianalytic solution $y=h(x)$, where $h(x)$ is allowed to have a certain controlled loss of regularity,
depending on $G$. Several important questions on quasianalytic functions, concerning division,
factorization, Weierstrass preparation, etc., fall into the framework of this problem (or are closely related),
and are also discussed.
\end{abstract}

\date{\today}
\maketitle
\setcounter{tocdepth}{1}
\tableofcontents

\section{Introduction}\label{sec:intro}
This article develops techniques for solving equations $G(x,y)=0$, where $G(x,y)=G(x_1,\ldots,x_n,y)$ is
a function in a given quasianalytic class (see Section 2). Assuming that $G(x,y)=0$ has a formal power
series solution $y=H(x)$ at some point $a$, we ask whether $H$ is the Taylor expansion at $a$ of a
quasianalytic solution $y=h(x)$, where $h(x)$ is allowed to have a certain controlled loss of regularity,
depending on $G$. Several important problems on quasianalytic functions, concerning division,
factorization, Weierstrass preparation, etc., fall into the framework of this question, or are closely related,
and they are also discussed in the paper.

There are two general categories of quasianalytic classes $\cQ$ that are studied in the recent literature:

\smallskip\noindent
(1)\, Quasianalytic Denjoy-Carleman classes $\cQ=\cQ_M$, going back to E. Borel \cite{Borel} and characterized
(following questions of Hadamard in studies of linear partial differential equations \cite{Had}) by the Denjoy-Carleman
theorem \cite{Den}, \cite{Carl}. These are classes of $\cC^\infty$ functions whose partial derivatives have bounds
on compact sets determined by a logarithmically convex sequence of 
positive real numbers $M = (M_j)_{j\in\IN}$; several
classical properities of $M$ guarantee that the functions of class $\cQ_M$ on an open subset $U$ of $\IR^n$ form
a ring $\cQ_M(U)$ that is closed under differentiation and (by the Denjoy-Carleman theorem)
\emph{quasianalytic} (i.e., the Taylor series homomorphism
at a point of $U$ is injective, if $U$ is connected). See \S\ref{subsec:DC}.

\smallskip\noindent
(2)\, Classes of $\cC^\infty$ functions that are definable in a given polynomially-bounded $o$-minimal
structure. Such structures arise in model theory, and define quasianalytic classes $\cQ$ according to a
result of C. Miller (see \cite{Mil}, \cite{RSW} and Remark \ref{rem:axioms}(3)).

\smallskip
Loss of regularity will be expressed as follows. If the equation $G(x,y)=0$ is of class $\cQ$, then a
solution $y=h(x)$ will be allowed to belong to a (perhaps larger) quasianalytic class $\cQ'$. For classes $\cQ$ as
in (2) above, we will always have $\cQ' = \cQ$. On the other hand, suppose that
$\cQ$ is a quasianalytic Denjoy-Carleman
class $\cQ_M$. Then we will find a positive integer $p$ depending on $G$, such that $h$ is of class $\cQ'$,
where $\cQ_M \subseteqq \cQ' \subseteqq \cQ_{M^{(p)}}$ and $M^{(p)}$ denotes the sequence $M^{(p)}_j := M_{pj}$.
More precisely, we can take $\cQ' = \cQ_{M^{(p)}} \bigcap \cC^{\infty}_M$ where $\cC^{\infty}_M(U)$ denotes
the subring of $\cC^\infty(U)$ of functions $f$ such that, for every relatively compact definable open $V\subset U$, $f|_V$ is
definable in the (polynomially bounded) $o$-minimal structure $\IR_{\cQ_M}$ generated by $\cQ_M$ (see \cite{RSW}
and \S\ref{subsec:shift}).

In the theorems following (proved in Sections  \ref{sec:main} and \ref{sec:poly}, respectively) and in
all results involving loss of regularity in Sections \ref{sec:contin}--\ref{sec:Weier}, $\cQ$ can be understood to mean a
quasianalytic class in one of the two general categories above, and then $\cQ'$ will mean either $\cQ$,
in the definable case (2), or $\cQ' \subseteqq \cQ_{M^{(p)}}$, as above, in the case that $\cQ=\cQ_M$ (1), 
where $p$ depends on
$G$. \emph{We fix this convention once and for all, and avoid repeating it in every
result.}

\begin{theorem}\label{thm:main}
Let $G(x,y)$ be a nonzero function of quasianalytic class $\cQ$, defined in a neighbourhood $U\times W$ of
$(a,b) \in \IR^n \times \IR$. Then there is a (perhaps larger) quasianalytic class $\cQ' \supseteqq \cQ$ 
such that,  if the equation
$$
G(x,y) = 0
$$
admits a formal power series solution $y=H(x)$ at the point $a$, with $b=H(a)$, then there is a solution 
$y = h(x) \in \cQ'(V)$, where $V$ is a neighbourhood of $a$ in $U$, and $H$ is the formal Taylor 
expansion of $h$ at $a$.
\end{theorem}

Of course, it is enough to find $h \in \cQ'(V)$ with formal Taylor expansion $H$ at $a$, 
since it follows that $G(x,h(x))=0$, by quasianalyticity.

In the case that $G(x,y)$ is a monic polynomial in $y$ with quasianalytic coefficients, there is
a result stronger than the above. Theorem \ref{thm:main} does not evidently reduce to the case
of a monic polynomial equation because of the lack of a Weierstrass preparation theorem 
in quasianalytic classes; see Section \ref{sec:Weier} (cf.\,\cite{Child}).

\begin{theorem}\label{thm:poly}
Let $\cQ$ denote a quasianalytic class.
Let $U$ denote a (connected) neighbourhood of the origin in $\IR^n$, with coordinates $x= (x_1,\ldots,x_n)$, and let
\begin{equation}\label{eq:poly}
G(x,y) = y^d + a_1(x)y^{d-1} + \cdots + a_d(x),
\end{equation}
where the coefficients $a_i \in \cQ(U)$. Let
$$
G(x,y) = \prod_{j=1}^k \left(y^{d_j} + B_{j1}(x)y^{d_j -1} + \cdots + B_{j,d_j}(x)\right)
$$
denote the irreducible factorization of $G(x,y)$ as an element of $\IR\llb x\rrb [y]$. Then there
is a (perhaps larger) quasianalytic class $\cQ' \supseteqq \cQ$ and a neighbourhood $V$ of $0$ in $U$,
such that each $B_{ji}$ is the formal Taylor expansion $\hb_{ji,0}$ at $0$ of an element $b_{ji} \in \cQ'(V)$, and
$$
G(x,y) = \prod_{j=1}^k \left(y^{d_j} + b_{j1}(x)y^{d_j -1} + \cdots + b_{j,d_j}(x)\right)
$$
in $\cQ'(V)[y]$.
\end{theorem}

These theorems and the other results in Sections \ref{sec:contin}--\ref{sec:Weier} are not, however, particular to 
quasianalytic classes of types (1),\,(2) above. For $G(x,y)=0$ of any given quasianalytic class $\cQ$, the
solutions $y=h(x)$ will be $\cC^\infty$ functions whose composites by a certain finite sequence $\s$ of blowings-up 
and power substitutions (depending on $G$) belong to $\cQ$. Such functions satisfy the axiom of quasianalyticity 
(see Definitions \ref{def:quasian}). In the case (2) of functions definable in a given polynomially-bounded 
$o$-minimal structure, such $\cC^\infty$ functions are evidently definable in the same structure, so we can take 
$\cQ' = \cQ$. In Section \ref{sec:reg}, we will show that, if $\cQ$ is a quasianalytic Denjoy-Carleman class $\cQ_M$, 
then such $\cC^\infty$ functions belong to 
$\cQ_{M^{(p)}}$, for some $p$ depending on the sequence $\s$.

Theorems \ref{thm:main}, \ref{thm:poly} and the related results are proved using techniques of 
\emph{quasianalytic continuation} that
are developed in Section \ref{sec:contin}. 
Quasianalyticity provides a generalization of the classical property of analytic continuation.
We use the axiom of quasianalyticity to show that,
if the formal Taylor expansion $\hf_a$ of a quasianalytic function $f$ at a 
given point $a$ is the composite $H\circ\hs_a$ of a formal power series $H$ with the formal 
expansion of a suitable quasianalytic mapping $\s$, then this formal composition property extends 
to a neighbourhood of $a$.
The main problems are solved by reducing $G(x,y)$ to a simpler form by composing
with an appropriate sequence of blowings-up and power substitutions, finding a solution of
the simpler problem, and using the quasianalytic continuation property to descend to a
solution of the original equation.

\section{Quasianalytic classes}\label{sec:quasian}

We consider a class of functions $\cQ$ given by the association, to every 
open subset $U\subset \IR^n$, of a subalgebra $\cQ(U)$ of $\cC^\infty (U)$ containing
the restrictions to $U$ of polynomial functions on $\IR^n$, and closed under composition 
with a $\cQ$-mapping (i.e., a mapping whose components belong to $\cQ$). 
We assume that $\cQ$ determines a sheaf of local $\IR$-algebras of $\cC^\infty$ functions on $\IR^n$,
for each $n$, which we also denote $\cQ$.

\begin{definition}[quasianalytic classes]\label{def:quasian}
We say that $\cQ$ is \emph{quasianalytic} if it satisfies the following three axioms:

\begin{enumerate}
\item \emph{Closure under division by a coordinate.} If $f \in \cQ(U)$ and
$$
f(x_1,\dots, x_{i-1}, a, x_{i+1},\ldots, x_n) = 0,
$$
where $a \in \IR$,  then $f(x) = (x_i - a)h(x),$ where $h \in \cQ(U)$.

\smallskip
\item \emph{Closure under inverse.} Let $\varphi : U \to V$
denote a $\cQ$-mapping between open subsets $U$, $V$ of $\IR^n$.
Let $a \in  U$ and suppose that the Jacobian matrix
$$
\frac{\partial \varphi}{\partial x} (a) := \frac{\partial
(\varphi_1,\ldots, \varphi_n) }{\partial (x_1,\ldots, x_n)}(a)
$$
is invertible. Then there are neighbourhoods $U'$ of $a$ and $V'$ of 
$b := \varphi(a)$, and a $\cQ$-mapping  $\psi: V' \to U'$ such that
$\psi(b) = a$ and $\psi\circ \varphi$  is the identity mapping of
$U '$.

\smallskip
\item \emph{Quasianalyticity.} If $f \in \cQ(U)$ has Taylor expansion zero
at $a \in U$, then $f$ is identically zero near $a$.
\end{enumerate}
\end{definition}

\begin{remarks}\label{rem:axioms} (1)\, Axiom \ref{def:quasian}(1) implies that, 
if $f \in \cQ(U)$, then all partial derivatives of $f$ belong to $\cQ(U)$. 

\smallskip\noindent
(2)\, Axiom \ref{def:quasian}(2) is equivalent to the property that the implicit function theorem holds for functions of 
class $\cQ$.  It implies that the reciprocal of a nonvanishing function of class $\cQ$ is also of class $\cQ$.

\smallskip\noindent
(3)\, In the case of $\cC^\infty$ functions definable in a given polynomially bounded
$o$-minimal structure, we can define a quasianalytic class $\cQ$ in the axiomatic
framework above by taking $\cQ(U)$ as the subring of $\cC^\infty(U)$ of functions
$f$ such that $f$ is definable in some neighbourhood of any point of $U$ (or, equivalently,
such that $f|_V$ is definable, for every relatively compact definable open $V\subset U$).
\end{remarks} 

The elements of a quasianalytic class $\cQ$ will be called \emph{quasianalytic functions}. 
A category of manifolds and mappings of class $\cQ$ can be defined in a standard way. The category 
of $\cQ$-manifolds is closed under blowing up with centre a $\cQ$-submanifold \cite{BMselecta}.

Resolution of singularities holds in a quasianalytic class \cite{BMinv}, \cite{BMselecta}. Resolution of
singularities of an ideal does not require that the ideal be finitely generated; see \cite[Thm.\,3.1]{BMV}.
Resolution of singularities of an ideal in a quasianalytic class is the main tool used in this article.

\subsection{Quasianalytic Denjoy-Carleman classes}\label{subsec:DC}
We use standard multiindex notation: Let $\IN$ denote the nonnegative integers. If $\al = (\al_1,\ldots,\al_n) \in
\IN^n$, we write $|\al| := \al_1 +\cdots +\al_n$, $\al! := \al_1!\cdots\al_n!$, $x^\al := x_1^{\al_1}\cdots x_n^{\al_n}$,
and $\p^{|\al|} / \p x^{\al} := \p^{\al_1 +\cdots +\al_n} / \p x_1^{\al_1}\cdots \p x_n^{\al_n}$. We write $(i)$ for the
multiindex with $1$ in the $i$th place and $0$ elsewhere.

\begin{definition}[Denjoy-Carleman classes]\label{def:DC}
Let $M = (M_k)_{k\in \IN}$ denote a sequence of positive real numbers which is \emph{logarithmically
convex}; i.e., the sequence $(M_{k+1} / M_k)$ is nondecreasing.
A \emph{Denjoy-Carleman
class} $\cQ = \cQ_M$ is a class of $\cC^\infty$ functions determined by the following condition: A function 
$f \in \cC^\infty(U)$ (where $U$ is open in $\IR^n$) is of class $\cQ_M$ if, for every compact subset $K$ of $U$,
there exist constants $A,\,B > 0$ such that
\begin{equation}\label{eq:DC}
\left|\frac{\p^{|\al|}f}{\p x^{\al}}\right| \leq A B^{|\al|} \al! M_{|\al|}
\end{equation}
on $K$, for every $\al \in \IN^n$.
\end{definition}

\begin{remarks}\label{rem:DC}
(1)\, The logarithmic convexity assumption implies that
$M_jM_k \leq M_0M_{j+k}$, for all $j,k$, and that
the sequence $((M_k/M_0)^{1/k})$ is nondecreasing.
The first of these conditions guarantees that $\cQ_M(U)$ is a ring, and 
the second that $\cQ_M(U)$ contains the ring $\cO(U)$ of real-analytic functions on $U$,
for every open $U\subset \IR^n$.
(If $M_k=1$, for all $k$, then $\cQ_M = \cO$.)

\smallskip\noindent
(2)\, $\cQ_M$ can be defined equivalently using inequalities of the form $|\p^{|\al|}f / \p x^{\al}| \leq
A B^{|\al|} |\al|! M_{|\al|}$, instead of \eqref{eq:DC}. This is true because, on the one hand, $\al! \leq |\al|!$, and,
on the other, $|\al|! \leq n^{|\al|} \al!$, since
$$
n^\al = (1+\cdots +1)^{\al_1+\cdots +\al_n} = \sum\frac{(\al_1+\cdots +\al_n)!}{\al_1!\cdots\al_n!},
$$
where the sum is over all partitions $|\al| = \al_1 +\cdots +\al_n$ of $|\al|$.
\end{remarks}

A Denjoy-Carleman class $\cQ_M$ is a quasianalytic class in the sense of Definition \ref{def:quasian}
if and only if the sequence
$M = (M_k)_{k\in \IN}$ satisfies the following two assumptions in addition to those
of Definition \ref{def:DC}.
\begin{enumerate}
\item[(a)] $\displaystyle{\sup \left(\frac{M_{k+1}}{M_k}\right)^{1/k} < \infty}$.

\smallskip
\item[(b)] $\displaystyle{\sum_{k=0}^\infty\frac{M_k}{(k+1)M_{k+1}} = \infty}$.
\end{enumerate}

It is easy to see that the assumption (a) implies that $\cQ_M$ is closed under differentiation.
The converse of this statement is due to S. Mandelbrojt \cite{Mandel}. In a Denjoy-Carleman class
$\cQ_M$, closure under differentiation is equivalent to the axiom \ref{def:quasian}(1) of closure under division by a
coordinate---the converse of Remark \ref{rem:axioms}(1) is a consequence of the fundamental
theorem of calculus:
\begin{equation}\label{eq:FTC}
f(x_1,\ldots,x_n) - f(x_1,\ldots, 0,\ldots, x_n) = x_i\int_0^1\frac{\p f}{\p x_i}(x_1,\ldots,tx_i,\ldots,x_n) dt
\end{equation}
(where $0$ in the left-hand side is in the $i$th place).

According to the Denjoy-Carleman theorem, the class $\cQ_M$ is quasianalytic (axiom \ref{def:quasian}(3))
if and only if the assumption (b) holds \cite[Thm.\,1.3.8]{Horm}.

Closure of the class $\cQ_M$ under composition is due to Roumieu \cite{Rou} and closure under
inverse to Komatsu \cite{Kom}; see \cite{BMselecta} for simple proofs. The assumptions 
of Definition \ref{def:DC} and (a),\,(b)
above thus guarantee that $\cQ_M$ is a quasianalytic class, in the sense of Definition \ref{def:quasian}.

If $\cQ_M$, $\cQ_N$ are Denjoy-Carleman classes, then $\cQ_M(U) \subseteq \cQ_N(U)$, for all $U$,
if and only if $\sup \left(M_k /N_k\right)^{1/k} < \infty$ (see \cite[\S1.4]{Th1}); in this case, we write
$\cQ_M \subseteq \cQ_N$.

\subsection{Shifted Denjoy-Carleman classes}\label{subsec:shift}
Given $M = (M_j)_{j\in \IN}$ and a positive integer $p$, let $M^{(p)}$ denote the sequence $M^{(p)}_j := M_{pj}$.

If $M$ is logarithmically convex, then $M^{(p)}$ is logarithmically convex:
$$
\frac{M_{kp}}{M_{(k-1)p}} = \frac{M_{kp}}{M_{kp-1}} \cdots \frac{M_{kp -p+1}}{M_{kp - p}} \leq
\frac{M_{kp+p}}{M_{kp+p-1}} \cdots \frac{M_{kp+1}}{M_{kp}} = \frac{M_{(k+1)p}}{M_{kp}}.
$$
Therefore, if $\cQ_M$ is a Denjoy-Carleman class, then so is $\cQ_{M^{(p)}}$. Clearly,
$\cQ_M \subseteqq \cQ_{M^{(p)}}$. Moreover, the assumption (a)
above for $\cQ_M$ immediately implies the same condition for $\cQ_{M^{(p)}}$. In general, however, it is not true 
that assumption (b) (i.e., the quasianalyticity axiom (3)) for $\cQ_M$ implies (b) for $\cQ_{M^{(p)}}$ \cite[Example 6.6]{Nelim}.

In particular, in general,
$\cQ_{M^{(p)}} \supsetneq \cQ_M$. Moreover, $\cQ_{M^{(2)}}$ is the smallest Denjoy-Carleman class containing
all $g \in \cC^\infty(\IR)$ such that $g(t^2)\in \cQ_M(\IR)$ \cite[Rmk.\,6.2]{Nelim}.

\section{Regularity estimates}\label{sec:reg}

Let $y=\s(x)$ denote a mapping of Denjoy-Carleman class $\cQ_M$ (Definition \ref{def:DC}). Given a $\cC^\infty$ function
$g(y)$ such $f(x) := g(\s(x))$ is of class $\cQ_M$, what can we say about the class of $g$? We will answer this
question for mappings $\s$ of two important kinds that will be needed for our main results: power substitutions
and blowings-up.

\subsection{Power substitutions}\label{sec:power} 
The following result is proved in \cite[Thm.\,6.1]{Nelim} 
for functions of a single variable, by an argument different from that below.

\begin{lemma}\label{lem:power}
Consider a Denjoy-Carleman class  $\cQ_M$.
Let $U = \prod_{i=1}^n (-r_i, r_i) \subset \IR^n$, where each $r_i > 0$, and let $\s: U \to V$ denote a 
\emph{power substitution}
\begin{equation*}
(y_1,\ldots,y_n) = (x_1^{k_1},x_2^{k_2},\ldots,x_n^{k_n}),
\end{equation*}
where each $k_i$ is a positive integer and $V = \prod (-r_i^{k_i}, r_i^{k_i})$.
Let $g \in \cC^{\infty}(V)$ and let $f = g\circ\s$. If $f \in \cQ_M(U)$, then 
$g \in \cQ_{M^{(p)}}(\sigma(U))$, where $p = \max k_i$.
\end{lemma}

\begin{proof}
Let $K \subset U$ denote the compact set $\prod_{i=1}^n [-r_i/2, r_i/2]$. Since $f \in \cQ_M(U)$, there are constants 
$A>0,\, B\geq 1$ such that
\begin{equation}\label{eq:power1}
\left| \frac{\p^{|\al|} f}{\p x^{\al}} \right| \leq   A B^{|\al|}\al!M_{|\al|}
\end{equation}
on the compact set $K$, for all $\al \in \IN^n$. We will show that 
\begin{equation}\label{eq:power2}
\left| \frac{\p^{|\be|} g}{\p y^{\be}}\right|  \leq  A B^{p|\beta|} \beta! M_{p |\beta|}
\end{equation}
on $\s(K)$, for all $\be \in \IN^n$. In the following, we will not explicitly 
write ``on $K$'' or ``on $\s(K)$''---all estimates will be understood
to mean on these sets (and the left-hand side of \eqref{eq:power1} or \eqref{eq:power2} will
sometimes be understood to mean the maximum on one of these sets, when the meaning is clear
from the context). We will use the notation
\begin{equation*}
g^{(\be)} := \frac{\p^{|\be|} g}{\p y^{\be}}.
\end{equation*}

\begin{claim}\label{claim:power}
For each $\beta\in \IN^n$, 
\begin{equation}\label{eq:claim}
\left| \frac{\p^{|\al|} (g^{(\be)}\circ\s)}{\p x^{\al}} \right| \leq   A B^{p|\be|+|\al|}(\al +\be)!M_{p|\be| +|\al|} \Ga(k,\al,\be),
\end{equation}
for all $\al \in \IN^n$, where
\begin{equation*}
\Ga(k,\al,\be) := \frac{\al!\prod_{i=1}^n \prod_{j=1}^{\be_i}(jk_i + \al_i)}{k^\be (\al + \be)!},
\end{equation*}
and $k := (k_1,\ldots,k_n)$, $k^\be := k_1^{\be_1}\cdots k_n^{\be_n}$.
\end{claim}

Note that $\Ga(k,0,\be) = \left(\prod_{i=1}^n \be_i! k_i^{\be_i}\right) / (k^\be \be!) = 1$.
In the case that $\al = 0$, therefore, 
\eqref{eq:claim} reduces to \eqref{eq:power2}, so the lemma follows from the claim.

We will prove Claim \ref{claim:power} by induction on $|\beta|$. Note that $\Ga(k,\al,0) = 1$. The claim is 
therefore true when $\beta = 0$, because in this case
\eqref{eq:claim} reduces to \eqref{eq:power1}.

Assume that \eqref{eq:claim} holds for a given multiindex $\beta$. It is clearly then enough to prove
\eqref{eq:claim} for $\ga := \be + (1)$. The partial derivative $\p/\p y_1$ transforms by $\s$ as follows:
\begin{equation}\label{eq:trans}
\frac{\p}{\p y_1} =  \frac{1}{ k_1 \, x_1^{k_1-1}} \frac{\p}{\p x_1};\qquad \text{i.e.,\,\, }
\frac{\p h}{\p y_1} \circ \s =  \frac{1}{ k_1 \, x_1^{k_1-1}} \frac{\p (h\circ\s)}{\p x_1}, \quad h\in \cC^\infty(V).
\end{equation}
Therefore,
\begin{equation*}
\frac{\p g^{(\be)}}{\p y_1}\circ\s = \frac{1}{k_1}\int_{[0,1]^{k_1-1}} \frac{\partial^{k_1} (g^{(\be)}\circ\s)}{\partial x_1^{k_1}}\left(t_1\cdots t_{k_1-1}x_1, x_2, \ldots, x_n\right) Q_0(t)\,dt,
\end{equation*}
by \eqref{eq:FTC} (applied $k_1-1$ times), where $t = (t_1,\ldots, t_{k_1-1})$ and $Q_0(t)$ denotes the polynomial
$Q_0(t) := t_1^{k_1-2} t_2^{k_1-3}\cdots t_{k_1-2}$. It follows that, for all 
$\alpha \in \IN^n$,
\begin{align*}
\left| \frac{\p^{|\al|} (g^{(\ga)}\circ\s)}{\p x^\al}\right| &= \left|\frac{\p^{|\al|}}{\p x^\al}\left(\frac{\p g^{(\be)}}{\p y_1}\circ\s\right)\right|
\\[1em]
&\leq \frac{1}{k_1} \left| \frac{\p^{|\al|+k_1} (g^{(\be)}\circ\s)}{\p x^{\al + k_1(1)}}\right|\cdot \int_{[0,1]^{k_1-1}} Q_\al(t)dt,
\end{align*}
where $Q_{\al}(t) := Q_0(t)(t_1\cdots t_{k_1-1})^{\al_1} = t_1^{k_1-2+\al_1}t_2^{k_1-3+\al_1}\cdots t_{k-1-1}^{\al_1}$,
so that
\begin{equation*}
\int_{[0,1]^{k_1-1}}  Q_{\al}(t) dt = \frac{1}{(k_1-1+\al_1)(k_1-2+\al_1)\cdots(1+\al_1)} = \frac{(k_1+\al_1) \al_1!}{(k_1+\alpha_1)!}.
\end{equation*}
By the induction hypothesis,
\begin{multline*}
\left| \frac{\p (g^{(\ga)}\circ\s)}{\p x^\al}\right| \leq A B^{p|\be| + |\al| + k_1} (\al + \ga)! M_{p|\be| + |\al| + k_1}\\
\cdot \frac{1}{k_1}\cdot \frac{(\al_1 + \be_1 + k_1)!}{(\al_1 + \be_1 + 1)!}\cdot \Ga(k,\al + k_1(1), \be)\cdot \frac{(k_1+\al_1) \al_1!}{(k_1+\alpha_1)!}.
\end{multline*}
Since $p|\be| + |\al| + k_1 \leq p|\ga| + |\al|$, we get
$$
\left| \frac{\p^{|\al|} (g^{(\ga)}\circ\s)}{\p x^\al}\right| \leq A B^{p|\ga| + |\al|} (\al + \ga)! M_{p|\ga| + |\al|}\Ga(k,\al,\ga),
$$
for all $\al \in \IN^n$, as required, using the following combinatorial identity, which can be easily checked:
$$
\frac{1}{k_1}\cdot \frac{(\al_1 + \be_1 + k_1)!}{(\al_1 + \be_1 + 1)!}\cdot \frac{(k_1+\al_1) \al_1!}{(k_1+\alpha_1)!}
\cdot\Ga(k,\al + k_1(1), \be)\, =\, \Ga(k,\al,\be + (1)).
$$
\end{proof}

\begin{remark}\label{rem:power}
It follows from Lemma \ref{lem:power} that the same conclusion holds for a mapping $\s$ of the form
$(y_1,\ldots,y_n) = (\ep_1 x_1^{k_1},\ldots, \ep_n x_n^{k_n})$, where each $\ep_i \in  \{-1,1\}$.
\end{remark}

\subsection{Blowing up}\label{sec:blup} 

\begin{lemma}\label{lem:blup}
Let $\cQ_M$ denote a Denjoy-Carleman class. 
Let $W$ be an open subset of $\IR^n$, and 
let $\s:Z \to W$ denote a blowing-up with centre a $\cQ_M$-submanifold of $W$.
Let $g \in \cC^{\infty}(W)$ and let $f = g\circ\s$. If $f \in \cQ_M(Z)$, then $g \in \cQ_{M^{(2)}}(W)$.
\end{lemma}

In contrast to the case of a power substitution, we do not actually know whether the loss of regularity 
is necessary in Lemma \ref{lem:blup}; it is interesting to ask whether $g\in \cQ_M(W)$.

\begin{proof}[Proof of Lemma \ref{lem:blup}]
Any point of $W$ has a coordinate neighbourhood $V$, such that $\s^{-1}(V)$ can be covered by
finitely many coordinate charts $U$, in each of which $\s$ is given by a mapping of the form
\begin{equation*}
(y_1,\ldots,y_n) = (x_1,\,x_1x_2,\ldots,x_1x_s,\,x_{s+1},\ldots,x_n),
\end{equation*}
where $2\leq s\leq n$. In the following, we will use $\s: U \to V$ to denote this mapping.

Let $K \subset U$ denote the compact set $\prod_{i=1}^n [-r_i, r_i]$, where each $r_i > 0$. 
Since $f \in \cQ_M(U)$, there are constants $A>0,\, B\geq 1$ such that
\begin{equation}\label{eq:blup1}
\left| \frac{\p^{|\al|} f}{\p x^{\al}} \right| \leq   A B^{|\al|}\al!M_{|\al|}
\end{equation}
on the compact set $K$, for all $\al \in \IN^n$. We will show that 
\begin{equation}\label{eq:blup2}
\left| \frac{\p^{|\be|} g}{\p y^{\be}}\right|  \leq  A (4s^2 rB^2)^{|\beta|} \beta! M_{2 |\beta|}
\end{equation}
on $\s(K)$, for all $\be \in \IN^n$, where $r := \max\{1,r_i\}$. (Recall the conventions following 
\eqref{eq:power2} above.) The lemma then follows.

\begin{claim}\label{claim:blup}
For each $\beta\in \IN^n$, 
\begin{equation}\label{eq:blup}
\left| \frac{\p^{|\al|} (g^{(\be)}\circ\s)}{\p x^{\al}} \right| \leq   A (2s^2r)^{\be_1} B^{p_{\be}(\al)} \D(\al,\be)M_{p_{\be}(\al)},
\end{equation}
for all $\al \in \IN^n$, where $p_{\be}(\al) :=  |\be|+|\al| +\sum_{i=1}^s \be_i$,
\begin{equation}\label{eq:delta}
\D(\al,\be) := \frac{\al_1!(\al+\be+\ga)! }{(\al_1 + \xi)!(\be_1-\xi)!},
\end{equation}
and
\begin{description}
\item[$\ga$] is chosen to maximize $(\al +\be +\de)!)$ over the set $I(\be)$ consisiting of all 
$\de = (0,\de_2,\ldots,\de_s,0,\ldots,0) \in \IN^n$ such that $|\de|=\be_1$,
\item[$\xi$] is chosen to minimize $(\al_1 + \eta)!(\be_1-\eta)!$ over the set $J(\be)$ of 
all $\eta \in \IN$ such that $0\leq \eta\leq \be_1$.
\end{description}
\end{claim}

Note that $p_{\be}(0) \leq 2|\be|$, and that
\begin{align*}
\D(0,\be)\, &=\, \be!\cdot \frac{(\be_1)!}{(\xi)!(\be_1-\xi)!} \cdot \frac{\prod_{i=2}^s(\be_i + \ga_i)!}{\be_1! \prod_{i=2}^s \be_i!}\\
               &\leq\, \be!\cdot \frac{(\beta_1)!}{(\xi)!(\beta_1-\xi)!}\cdot\prod_{i=2}^s \frac{(\gamma_i + \beta_i)!}{\gamma_i!\cdot \beta_{i}!}\,
               \leq\, \be!\cdot 2^{\beta_1} \prod_{i=2}^s 2^{\gamma_i + \beta_i}
\end{align*}
(using Remark \ref{rem:DC}(2)). Therefore, \eqref{eq:blup} in the case that $\al=0$ implies \eqref{eq:blup2}; i.e.,
the lemma follows from Claim \ref{claim:blup}.

We will prove the claim by induction on $|\beta|$. Note that $\D(\al,0) = \al!$.
The claim is therefore true when $\beta = 0$, because in this case
\eqref{eq:blup} reduces to \eqref{eq:blup1}. Fix a multiindex $\tbe$, where $|\tbe| > 0$.
By induction, we assume the claim holds for all $\be$ such that $|\be| < |\tbe|$. Now,
the partial derivatives transform by $\s$ as follows (cf. \eqref{eq:trans}):
\begin{equation}\label{eq:transfs}
\begin{aligned}
\frac{\p}{\p y_1} &=  \frac{\p}{\p x_1}  - \sum_{j=2}^s \frac{x_j}{x_1} \frac{\p}{\p x_j}, &&\\
\frac{\p}{\p y_i}  &=  \frac{1}{x_1} \frac{\p}{\partial x_i},  && i=2,\ldots,s,\\
\frac{\p}{\p y_i} &= \frac{\p}{\partial x_i}, && i=s+1, \ldots, n.
\end{aligned}
\end{equation}

\medskip\noindent
\emph{Case 1.} $\tbe_1=0$. Then there exists $\be \in \IN^n$ such that $\tbe = \be + (k)$, where $2\leq k\leq n$.
Since $\be_1 = 0$, \eqref{eq:blup} holds for all $\al \in \IN^n$, with $\ga =0$ and $\xi =0$ in $\D(\al,\be)$.
If $k > s$, then \eqref{eq:blup} for $\tbe$ follows from \eqref{eq:transfs} and the inductive assumption
\eqref{eq:blup} for $\be$.

On the other hand, suppose that $2\leq k \leq s$. Then
\begin{equation}\label{eq:case1}
\frac{\p g^{(\be)}}{\p y_k}\circ\s = \int_0^1 \frac{\p^2 (g^{(\be)}\circ\s)}{\p x_1 \p x_k}\left(t x_1, x_2, \ldots, x_n\right)dt.
\end{equation}
Given any $\al \in \IN^n$, let $\de := \al +(1)+(k)$. Then $p_\be(\de) = p_{\tbe}(\al)$ and 
$$
\D(\de,\be) = (\de + \be)! = (\al_1 +1)(\al + \tbe)!= (\al_1 +1)\D(\al,\tbe).
$$
By \eqref{eq:case1} and the induction hypothesis,
\begin{align*}
\left|\frac{\p^{|\al|}(g^{(\tbe)}\circ\s)}{\p x^{\al}}\right| 
    &= \left|\int_0^1 \frac{\p^{|\al|+2} (g^{(\be)}\circ\s)}{\p x^\al \p x_1 \p x_k}\left(t x_1, x_2, \ldots, x_n\right)t^{\al_1}dt\right|\\
    &\leq A B^{p_{\be}(\de)} \D(\de,\be) M_{p_{\be}(\de)}\cdot \frac{1}{\al_1 +1}\\
    &= A B^{p_{\tbe}(\al)} \D(\al,\tbe) M_{p_{\tbe}(\al)},
\end{align*}
as required.

\medskip\noindent
\emph{Case 2.} $\tbe_1>0$. Then there exists $\be \in \IN^n$ such that $\tbe = \be + (1)$, and
\begin{equation*}
\frac{\p g^{(\be)}}{\p y_1}\circ\s\, =\, \frac{\p (g^{(\be)}\circ\s)}{\p x_1} 
- \sum_{j=2}^s \int_0^1 x_j \frac{\p^2 (g^{(\be)}\circ\s)}{\p x_1 \p x_j}\left(t x_1, x_2, \ldots, x_n\right)dt,
\end{equation*}
by \eqref{eq:transfs}. Therefore, for all $\al \in \IN^n$,
\begin{equation*}
\left|\frac{\p^{|\al|}(g^{(\tbe)}\circ\s)}{\p x^{\al}}\right|\, \leq\, \RN{1} +\sum_{j=2}^s \RN{2}_j +\sum_{j=2}^s \RN{3}_j,
\end{equation*}
where
\begin{equation*}
\begin{aligned}
\RN{1}\, &:=\, \left| \frac{\p^{|\al|+1} (g^{(\be)}\circ\s)}{\p x^{\al} \p x_1}\right|,\\
\RN{2}_j\, &:=\, \int_0^1 \al_j \left|\frac{\p^{|\al|+1} (g^{(\be)}\circ\s)}{\p x^\al \p x_1}\left(t x_1, x_2, \ldots, x_n\right)\right|t^{\al_1}dt,\\
\RN{3}_j\, &:=\, \int_0^1 \left| x_j \cdot\frac{\p^{|\al|+2} (g^{(\be)}\circ\s)}{\p x^\al \p x_1 \p x_j}\left(t x_1, x_2, \ldots, x_n\right)\right|t^{\al_1}dt.
\end{aligned}
\end{equation*}

We will use the inductive hypothesis to show that each term $\RN{1},\,\RN{2}_j$ and $\RN{3}_j$ is bounded by
\begin{equation}\label{eq:bound}
Asr (2s^2r)^{\tbe_1 -1} B^{p_{\tbe}(\al)} \D(\al,\tbe) M_{p_{\tbe}(\al)}.
\end{equation}
The required estimate \eqref{eq:blup} for $\tbe$ follows since there are only $2s-1 \leq 2s$ such terms.

Consider the first term $\RN{1}$. Set $\de := \al + (1)$. Then $p_\be(\de) = p_{\tbe}(\al) - 1$. Choose $\ga \in I(\be),\,
\xi \in J(\be)$ to realize $\D(\de,\be)$ according to the formula \eqref{eq:delta} (with $\de$ in place of $\al$).
If $\tbe_1 = 1$, then $\be_1 =0$, so that $\ga=0,\, \xi=0$; in this case, it is easy to see that
$$
\D(\de,\be) = (\de + \be)! = (\al + \tbe)! \leq \D(\al, \tbe).
$$

On the other hand, suppose that $\tbe_1 > 1$. Since $|\ga| = \tbe_1 - 1$, there exists $k \in \{2,\ldots,s\}$ such
that $\ga_k \geq (\tbe_1 - 1)/(s-1) > 0$. Since $\ga_k \in \IN$, it follows that, in fact, $\ga_k \geq \tbe_1 / s$
(consider separately the cases that $\tbe_1 > s$ or $\tbe_1 \leq s$). Set $\ga' := \ga + (k)$. Then $\ga' \in J(\tbe)$ and
$$
\D(\de,\be)\, =\, \frac{\al_1 + 1}{\al_k+\be_k+\ga_k+1}\cdot \frac{\al_1! (\al +\tbe +\ga')!}{(\al_1+1+\xi)!(\tbe_1-1-\xi)!}
$$
Now,
$$
 \frac{\al_1 + 1}{\al_k+\be_k+\ga_k+1}\, \leq\, \frac{s(\al_1+1)}{\tbe_1}\, \leq\, \frac{s(\al_1+1+\xi)}{\tbe_1-\xi}.
$$
Since $\xi \in J(\tbe)$, we get $\D(\de,\be)\leq s\D(\al,\tbe)$.

In either case, by the induction hypothesis,
\begin{align*}
\RN{1}\, &\leq\, A (2s^2r)^{\be_1} B^{p_{\be}(\de)} \D(\de,\be) M_{p_{\be}(\de)}\\
            &\leq\, A s (2s^2r)^{\tbe_1 -1} B^{p_{\tbe}(\al)} \D(\al,\tbe) M_{p_{\tbe}(\al)}.
\end{align*}

Secondly, consider a term $\RN{2}_j$. We can assume that $\al_j\neq 0$. Again set $\de := \al + (1)$ (so that 
$p_\be(\de) = p_{\tbe}(\al) - 1$), and choose
$\ga \in I(\be),\, \xi \in J(\be)$ to realize $\D(\de,\be)$. Set $\ga' := \ga + (j)$. Then $\ga' \in I(\tbe)$, and
\begin{align*}
\D(\de,\be)\, &=\, \frac{\al_1 + 1}{\al_j+\be_j+\ga_j+1}\cdot \frac{\al_1! (\al +\tbe +\ga')!}{(\al_1+1+\xi)!(\tbe_1-1-\xi)!}\\
                  &\leq\, \frac{\al_1 + 1}{\al_j} \cdot\D(\al,\tbe).
\end{align*}
By the induction hypothesis,
\begin{align*}
\RN{2}_j\, &\leq\, \al_j A (2s^2r)^{\be_1} B^{p_{\be}(\de)} \D(\de,\be) M_{p_{\be}(\de)}\cdot\frac{1}{\al_1+1}\\
            &\leq\, A (2s^2r)^{\tbe_1 -1} B^{p_{\tbe}(\al)} \D(\al,\tbe) M_{p_{\tbe}(\al)}.
\end{align*}

Finally, consider any of the terms $\RN{3}_j$. Set $\de := \al + (1) + (j)$. Then $p_\be(\de) = p_{\tbe}(\al)$.
Choose $\ga \in I(\be),\, \xi \in J(\be)$ to realize $\D(\de,\be)$. Set $\ga' := \ga + (j)$. Then $\ga' \in I(\tbe)$, and
\begin{align*}
\D(\de,\be)\, &=\, \frac{\al_1 + 1}{\al_j+\be_j+\ga_j+1}\cdot \frac{\al_1! (\al +\tbe +\ga')!}{(\al_1+1+\xi)!(\tbe_1-1-\xi)!}\\
                  &\leq\, (\al_1 + 1) \cdot\D(\al,\tbe).
\end{align*}
By the induction hypothesis,
\begin{align*}
\RN{3}_j\, &\leq\, r A (2s^2r)^{\be_1} B^{p_{\be}(\de)} \D(\de,\be) M_{p_{\be}(\de)}\cdot\frac{1}{\al_1+1}\\
            &\leq\, A r(2s^2r)^{\tbe_1 -1} B^{p_{\tbe}(\al)} \D(\al,\tbe) M_{p_{\tbe}(\al)}.
\end{align*}

Since $r,\,s \geq 1$, each of the terms $\RN{1},\,\RN{2}_j$ and $\RN{3}_j$ is bounded by \eqref{eq:bound},
and the proof is compete.
\end{proof}

\section{Quasianalytic continuation}\label{sec:contin}
Let $\cF_a$ denote the ring of formal power series centred at a point $a\in \IR^n$; thus 
$\cF_a \cong \IR\llb x_1,\ldots,x_n\rrb$. If $U$ is open in $\IR^n$ and $f \in \cC^\infty(U)$,
then $\hf_a \in \cF_a$ denotes the formal Taylor expansion of $f$ at a point $a \in U$; i.e.,
$\hf_a(x) = \sum_{\al\in\IN^n}(\p^{|\al|}f/\p x^{\al})(a)x^\al/\al!$ (likewise
for a $\cC^\infty$ mapping $U \to \IR^m$).

Let $\cQ$ denote a quasianalytic class (Definition \ref{def:quasian}).

\begin{theorem}\label{thm:contin}
Let $U,\,V$ denote open neighbourhoods of the origin in $\IR^n$, with coordinate systems
$x= (x_1,\ldots,x_n),\, y=(y_1,\ldots,y_n)$, respectively. (Assume $U$ is chosen so that every
coordinate hyperplane $(x_i = 0)$ is connected). Let $\s:U \to V$ denote a $\cQ$-mapping such
that the Jacobian determinant $\det (\p \s / \p x)$ is a monomial times
an invertible factor in $\cQ(U)$. Let $f \in \cQ(U)$ and let $H \in \cF_0$ be a formal power series
centred at $0\in V$, such that $\hf_0 = H\circ\hs_0$. Then, for all $\be \in \IN^n$, there exists
$f_\be \in \cQ(U)$ such that $f_0 = f$ and
\begin{enumerate}
\item for all $a\in U$, $\hf_a = H_a\circ \hs_a$, where $H_a \in \cF_{\s(a)}$ denotes the formal power series
\begin{equation}\label{eq:contin}
H_a := \sum_{\be \in \IN^n} \frac{f_\be(a)}{\be !}y^\be;
\end{equation}
\item each $f_\be$, $\be \in \IN^n$, and therefore also $H_a \in \cF_{\s(a)}$ (as a function of $a$)
is constant on connected components of the fibres of $\s$.
\end{enumerate}
\end{theorem}

\begin{proof}
(1)\, Write $\s = (\s_1,\ldots,\s_n)$ with respect to the coordinates of $V$. As formal expansions
at $0 \in U$,
$$
\sum_{j=1}^n \left(\frac{\p H}{\p y_j}\circ\s\right)\cdot\frac{\p\s_j}{\p x_i} = \frac{\p f}{\p x_i},
\quad i = 1,\ldots,n,
$$
so that
$$
\det\left(\frac{\p\s}{\p x}\right)\cdot\left(\frac{\p H}{\p y_j}\circ \s\right) 
= \left(\frac{\p\s}{\p x}\right)^{*}\left(\frac{\p f}{\p x_i}\right),
$$
where $(\p f/ \p x_i)$ denotes the column vector with components $\p f/ \p x_i$, and
$(\p\s/\p x)^*$ is the adjugate matrix of $\p\s/ \p x$.

By axioms \ref{def:quasian}(1),\,(3), for each $j=1,\dots,n$, there is a quasianalytic function $f_{(j)} \in \cQ(U)$
such that 
$$
\hf_{(j),0} = \frac{\p H}{\p y_j}\circ \hs_0
$$
and
$$
\det\left(\frac{\p\s}{\p x}\right)\cdot\left(f_{(j)}\right) = \left(\frac{\p\s}{\p x}\right)^*\left(\frac{\p f}{\p x_i}\right)
$$
in $\cQ(U)$. 

It follows by induction on the order of differentiation that, for each $\be \in \IN^n$, there
is a quasianalytic function $f_\be \in \cQ(U)$ such that
$$
\hf_{\be,0} = \frac{\p^{|\be|} H}{\p y^\be}\circ \hs_0
$$
and
$$
\det\left(\frac{\p\s}{\p x}\right)\cdot\left(f_{\be+(j)}\right) 
= \left(\frac{\p\s}{\p x}\right)^*\left(\frac{\p f_\be}{\p x_i}\right). 
$$
Therefore, for all $a \in U$, $\hf_a = H_a\circ \hs_a$, 
where $H_a$ is the formal power
series centred at $\s(a)\in V$ given by \eqref{eq:contin}. Likewise, for all $\be \in \IN^n$ and $a\in U$,
\begin{equation}\label{eq:higherord}
\hf_{\be,a} = \frac{\p^\be H_a}{\p y^\be}\circ \hs_a
\end{equation}

\smallskip
\noindent
(2)\, It is enough to show that, for each $\be \in \IN^n$, $f_\be$ is locally constant on every
fibre of $\s$. This is immediate from Lemma \ref{lem:contin} following applied at any given point 
$a \in U$ to the equation \eqref{eq:higherord}.
\end{proof}

\begin{lemma}\label{lem:contin}
Let $\s:U \to V$ denote a $\cQ$-mapping, where $U,\,V$ are open neighbourhoods of the origin 
in $\IR^n$. Let $f \in \cQ(U)$ and let $H \in \cF_0$ be a formal power series
centred at $0\in V$, such that $\hf_0 = H\circ\hs_0$. Then there is a neighbourhood $W$ of $0$ in $U$
such that $f$ is constant on the fibres of $\s$ in $W$.
\end{lemma}

\begin{proof}
The following argument is due to Nowak \cite{Ndiv}. We can assume that $f(0) =0$, $H(0) = 0$.
Let
\begin{equation}\label{eq:P}
P:= \{(\xi, \eta, \zeta) \in U \times U \times V: \s(\xi) = \zeta = \s(\eta),\, f(\xi) \neq f(\eta)\}.
\end{equation}
Suppose the lemma is false. Then $(0,0,0) \in \overline{P}$. By the quasianalytic curve selection
lemma (see \cite[Thm.\,6.2]{BMselecta}), there is a quasianalytic arc $(\al(t), \be(t), \ga(t)) \in U\times U\times V$
such that $(\al(0), \be(0), \ga(0)) = (0,0,0)$ and $(\al(t), \be(t), \ga(t)) \in P$ if $t\neq 0$. 
Then
$$
(f\circ\al)^{\wedge}_0 = \hf_0\circ\hal_0 = H\circ\hs_0\circ \hal_0 
= H\circ(\s\circ\al)^{\wedge}_0 = H\circ\hga_0.
$$
Likewise, $(f\circ\be)^{\wedge}_0 = H\circ\hga_0$, so that $(f\circ\al)^{\wedge}_0 
= (f\circ\be)^{\wedge}_0$. Since
$f\circ\al,\, f\circ\be$ are quasianalytic functions of $t$, $f\circ\al = f\circ\be$; a contradiction.
\end{proof}

\begin{remark}\label{rem:fibre} 
Theorem \ref{thm:contin}(2) also follows from Proposition \ref{prop:fibre} below, which is included 
here for completeness.
Proposition \ref{prop:fibre}
in the special case that $\cQ = \cO$ (the class of analytic functions) is proved in \cite[Prop.\,11.1]{BMrel2},
but the proof in the latter does not apply to quasianalytic classes, in general. We have chosen to prove
Theorem \ref{thm:contin}(2) using Lemma \ref{lem:contin} because the idea of its proof above will be 
needed again in our proof of Corollary \ref{cor:contin}.
\end{remark}

\begin{lemma}\label{lem:glue}
Let $\s: M \to V$ denote a proper $\cQ$-mapping, where $M$ is a $\cQ$-manifold of dimension
$n$, and $V$ is an open neighbourhood of the origin in $\IR^n$. Then, given any open covering $\{U\}$
of $\s^{-1}(0)$, there is a neighbourhood $W$ of $0$ in $V$ with the following properties:
\begin{enumerate}\item $\s^{-1}(W) \subset \bigcup U$.
\item Let $H \in \cF_0$ be a power series centred at $0 \in V$, and suppose there exists 
$f_U \in \cQ(U)$, for each $U$, such that $\hf_{U,a} = \hs_{a}^*(H)$, for all $a \in \s^{-1}(0)\cap U$. 
Then there exists $f \in \cQ(\s^{-1}(W))$ such that
$\hf_a = \hs^*_a(H)$, for all $a \in \s^{-1}(0)$.
\end{enumerate}
\end{lemma}

\begin{proof}
There is a covering of the fibre $\s^{-1}(0)$ by finitely many open sets $\Om_i$ with compact
closure, such that, for each $i$, there exists $U$ such that $\overline{\Om_i} \subset U$;
write $f_U = f_i$ (of course, $U$ is not necessarily unique). We can assume that each 
$\Om_i \bigcap \Om_j$ has only finitely many connected components (e.g., take each $\Om_i$
sub-quasianalytic).

For each $i$ and $j$, if $\Om$ is a connected component of $\Om_i \bigcap \Om_j$ and its closure
$\overline{\Om}$ includes a point of $\s^{-1}(0)$, then $f_i = f_j$ in $\Om$, by quasianalyticity.
For each $i$, let $V_i$ denote the complement in $\Om_i$ of the union of the $\overline{\Om}$,
for all connected components $\Om$ of $\Om_i \bigcap \Om_j$, for every $j$, such that 
$\overline{\Om} \bigcap  \s^{-1}(0) = \emptyset$. Then $\{V_i\}$ is an open covering
of $\s^{-1}(0)$. Let $W$ be any neighbourhood $0$ in $V$ such that $\s^{-1}(W) \subset \bigcup V_i$.
Then $f_i = f_j$ in $V_i \cap V_j \cap \s^{-1}(W)$, for all $i,j$, so the result follows.
\end{proof}

\begin{corollary}\label{cor:contin}
Let $\s: M \to V$ denote a proper $\cQ$-mapping, where $M$ is a $\cQ$-manifold of dimension
$n$, and $V$ is an open neighbourhood of the origin in $\IR^n$. Let $H \in \cF_0$ be a power
series centred at $0 \in V$. Suppose that, for all $a \in \s^{-1}(0)$, there is a neighbourhood $U$
of $a$ with coordinates $(x_1,\ldots,x_n)$ such that $\s|_U$ satisfies the hypotheses of Theorem \ref{thm:contin}, and there exists $f_U \in \cQ(U)$ such that $\hf_{U,a} = \hs_a^*(H)$. Then:
\begin{enumerate}
\item There is a neighbourhood $W$ of $0$ in $V$ such that $\s^{-1}(W) \subset \bigcup U$, and a function
$f_\be \in \cQ(\s^{-1}(W))$, for every $\be \in \IN^n$, with the following properties: each point of $\s^{-1}(0)$ has
a neighbourhood $\Om$ in $U \bigcap \s^{-1}(W)$, for some $U$, such that $f_{\be} = f_{U,\be}$ in $\Om$,
for all $\be$ (where $f_{U,\be}$ denotes the function associated to $f_U$ given by Theorem \ref{thm:contin}).
\item (After perhaps shrinking $W$) $f = f_0$ is \emph{formally composite} with $\s$; i.e.,
for all $b \in W$, there exists $H_b \in \cF_b$ such that $\hf_a = \hs^*_a(H_b)$, for all $a \in \s^{-1}(b)$.
\item In fact, there is a $\cC^\infty$ function $h \in \cC^\infty(W)$ such that $f = h\circ\s$.
\end{enumerate}
\end{corollary}

\begin{proof} (1)\,This gluing condition is immediate from Lemma \ref{lem:glue}.

\smallskip\noindent
(2)\, It is enough to show that, after shrinking $W$ if necessary, each $f_\be$ is constant on the
fibres of $\s$ over $W$. For every $k \in \IN$, let 
$$
P_k := \{(\xi,\eta,\zeta) \in M\times M\times W: \s(\xi) = \zeta = \s(\eta),\, f_\be(\xi) = f_\be(\eta),\, |\be|\leq k\}.
$$
Then the decreasing sequence of closed quasianalytic sets $P_0 \supset P_1 \supset P_2 \supset \cdots$
stabilizes in some neighbourhood of the compact set $\s^{-1}(0) \times \s^{-1}(0) \times \{0\}$, by
topological noetherianity \cite[Thm.\,6.1]{BMselecta}; say, $P_k = P_{k_0}$, $k\geq k_0$, in such a
neighbourhood. It follows that, if $W$ is a sufficiently small neighbourhood of $0$ in $V$, and $f_\be$ is
constant on the fibres of $\s$ over $W$, for all $\be \leq k_0$, then $f_\be$ is
constant on the fibres of $\s$ over $W$, for all $\be$.

Therefore, it is enough to prove the following assertion: given $\be \in \IN^n$, there is an open neighbourhood 
$W$ of $0$ such that $f_\be$ is constant on the fibres of $\s$ over $W$.
We can now argue as in the proof of Lemma \ref{lem:contin}.
Define $P \subset M \times M \times W$ as in \eqref{eq:P}. Suppose the assertion is false. Then
there is a point $(a_1,a_2.0) \in \overline{P}$, and a quasianalytic arc 
$(\al(t), \be(t), \ga(t)) \in M\times M\times V$
such that $(\al(0), \be(0), \ga(0)) = (a_1,a_2,0)$ and $(\al(t), \be(t), \ga(t)) \in P$ if $t\neq 0$. We
get a contradiction as before.

\smallskip\noindent
(3)\, The hypotheses on $\s$ imply that $\s$ is generically a submersion, so the assertion follows
from (1) by a quasianalytic generalization \cite{Bib}, \cite{Ncomp} of Glaeser's composite function theorem \cite{G}
(cf. Corollary \ref{cor:quasiancontin} ff. below).
\end{proof}

\begin{proposition}\label{prop:fibre} Let $\vp: M \to \IR^n$ denote a $\cQ$-mapping, where $M$ is a $\cQ$-manifold
of dimension $m$. Let $f \in \cQ(M)$ and let $H$ denote a formal power series at $b = 0 \in \IR^n$. Then
$$
S := \{a \in \vp^{-1}(b): \hf_a = H\circ\hvp_a\}
$$ 
is open and closed in $\vp^{-1}(b)$.
\end{proposition}

\begin{proof}
We work in a local coordinate chart of $M$ with coordinates $u = (u_1,\ldots,u_m)$ at a point $a=0$ in $\vp^{-1}(0)$. 
Write $H= \sum_{\be \in \IN^{n}} H_\be v^\be /\be!$, where $v = (v_1,\ldots,v_n)$. For $x \in \vp^{-1}(0)$,
\begin{align*}
\hf_x(u) - (H\circ\hvp_x)(u) &= \sum_{\al \in \IN^m} \frac{\p^\al f(x)}{\al!}u^\al 
                               - \sum_{\be \in \IN^n} \frac{H_\be}{\be!}\left(\sum_{|\al| > 0} \frac{\p^\al\vp(x)}{\al!}u^\al\right)^\be\\
                                           &= \sum_{\al \in \IN^m} \frac{1}{\al!} \left(\p^\al f(x) - K_\al(x)\right)u^\al,
\end{align*}
where each $K_\al(x)$ is a finite linear combination of products of derivatives of the components of $\vp$ (defined
in the coordinate neighbourhood). (We write $\p^{\al} := \p^{|\al|}/\p x^{\al}$ in this proof, and use the same notation
for the formal derivative of a power series, below.) If $x \in \vp^{-1}(0)$, then $\hf_x = H\circ\hvp_x$ if and only if 
$\p^\al f(x) - K_\al(x) = 0$, for all $\al$; i.e., $S$ is closed.

To show that $S$ is open, it is enough to prove that, if $a=0 \in S$ (i.e., $\hf_0 - H\circ\hvp_0 =0$),
then $\p^\al f - K_\al$ vanishes on a (common) neighbourhood of $a$ in $\vp^{-1}(0)$, for all $\al$; i.e.,
$(\p^\al f)(\ga(t)) - K_\al(\ga(t)) = 0$, for all $\al$, for any quasianalytic curve $\ga(t)$ in $\vp^{-1}(0)$, $\ga(0) = 0$.

Consider such a curve $\ga(t)$. Since $\hf_0 - H\circ\hvp_0 =0$ and $\vp\circ\ga = 0$,
\begin{align*}
0 &= \hf_0(\hga_0(t) + u) - H(\hvp_0(\hga_0(t) + u))\\
   &= \hf_0(\hga_0(t) + u) - H(\hvp_0(\hga_0(t) + u) - \hvp_0(\hga_0(t)))\\
   &= \sum_{\al} \frac{\p^\al \hf_0(\hga_0(t))}{\al!}u^\al - \sum_{\al} \frac{\wK_{\al,0}(\hga_0(t))}{\al!}u^\al;
\end{align*}
i.e., for all $\al$, $(\p^\al f\circ\ga)^\wedge_0(t) - (K_\al\circ\ga)^\wedge_0(t) = 0$; therefore,
$(\p^\al f)(\ga(t)) - K_\al(\ga(t)) = 0$, by quasianalyticity.
\end{proof}

\begin{remarks}\label{rem:maps}
(1)\, We will use the results above (apart from Proposition \ref{prop:fibre})
in the case that $\s$ is a quasianalytic mapping of one of two kinds:
\begin{enumerate}
\item[(a)] $\s: M \to V$ is a \emph{blowing-up} of $V$ with centre a closed submanifold of $V$ of class
$\cQ$, or, more generally, $\s$ is a finite composite of admissible blowings-up, where a blowing-up
is called \emph{admissible} is its center is a $\cQ$-manifold that has
only normal crossings with respect to the exceptional divisor.
\smallskip
\item[(b)]$\s: U \to V$ is a \emph{power substitution}
$$
(y_1,\ldots,y_n) = (x_1^{k_1},\ldots,x_n^{k_n}),
$$
where the exponents $k_i$ are positive integers, or, more generally, 
$$
\s: \coprod\limits_{\ep \in \{-1,1\}^n} U^\ep \to V,
$$
where $\coprod$ means disjoint union, each $U^\ep$, $\ep = (\ep_1,\ldots,\ep_n) \in \{-1,1\}^n$,
is a copy of $U$, and $\s^\ep := \s|_{U^\ep}$ is given by
$$
(y_1,\ldots,y_n) = (\ep_1 x_1^{k_1},\ldots, \ep_n x_n^{k_n}).
$$
(We assume that $V$ is of the form $\prod_{i=1}^n (-\de_i, \de_i)$, where each $\de_i > 0$.
The images $\s^\ep(U^\ep)$ are unions of closed quadrants, covering $V$.)
\end{enumerate}
\noindent
(2)\, Corollary \ref{cor:contin} extends in an immediate way to the case that, instead of
$\s: M \to V$, we have a locally finite covering $\{\s_j: M_j \to V\}$ of $V$ by quasianalytic
mappings, where each $\s_j$ is a finite composite of admissible local blowings-up. This
version of the corollary will be needed in the proof of Theorem \ref{thm:main}.

A family of mappings $\{\s_j: M_j \to V\}$ is a \emph{locally finite covering} of $V$ if (a) the images
$\s_j(M_j)$ are subordinate to a locally finite covering of $V$ by open subsets; (b) if $K$ is a
compact subset of $V$, then there is a compact subset $K_j$ of $M_j$, for each $j$, such
that $K = \bigcup\s_j(K_j)$ (the union is finite, by (a)). A \emph{local blowing-up} of $V$ is a blowing-up
over an open subset of $V$.
\end{remarks}

\begin{corollary}\label{cor:quasiancontin}
Assume in Corollary \ref{cor:contin} that $\s$ is a mapping of either kind in Remarks \ref{rem:maps}(1).
Then there is a (perhaps larger) quasianalytic class $\cQ'$ depending only on $\cQ$ and $\s$,
and there exists $h \in \cQ'(W\bigcap\s(U))$ such that $f = h\circ\s$. Likewise for the version of Corollary
\ref{cor:contin} given in Remarks \ref{rem:maps}(2).
\end{corollary}

This is a consequence of Corollary \ref{cor:contin} and 
Lemmas \ref{lem:power}, \ref{lem:blup} (see also Remark \ref{rem:power}). 
Note that, to prove Corollary \ref{cor:quasiancontin},
we need to use only Glaeser's original theorem \cite{G} (rather than a quasianalytic version) 
in the proof of Corollary \ref{cor:contin}(3), because blowings-up or power substitutions
are algebraic (polynomial) mappings (with respect to suitable quasianalytic coordinates).

We illustrate the use of the techniques above in two special cases of our main theorems:

\begin{proposition}[Membership in a principal ideal; cf. \cite{Ndiv}]\label{prop:princ}
Let $\cQ$ be a quasianalytic class and let $g \in \cQ(V)$, where $V$ is a neighbourhood of 
$0$ in $\IR^n$. Then there is a quasianalytic class $\cQ' \supseteqq \cQ$ such that,
given $f \in \cQ(V)$ and a formal power series $H \in \cF_0$ such that $\hf_0 = H\cdot \hg_0$, 
there exists $h \in \cQ'(W)$, where $W$ is a neighbourhood of $0$ in $V$, such that $f = h\cdot g$.
\end{proposition}

\begin{remarks}\label{rem:princ} 
(1)\, Thilliez has studied several cases of functions $g \in \cQ(V)$, where $\cQ$ is a quasianalytic Denjoy-Carleman
class $\cQ_M$, which satisfy the following property: if $f \in \cQ(V)$ 
and $\hf_a$ is divisible by $\hg_a$, for all $a\in V$, then $f=h\cdot g$, where $h \in \cQ(V)$ (see \cite[Section 4]{Th1} 
and Remark \ref{rem:Weier}(2) below). By 
Proposition \ref{prop:princ}, formal divisibility at a single point implies formal divisibility throughout a neighbourhood.

\smallskip\noindent
(2)\, If $f$ is merely $\cC^\infty$, the latter statement is not true, and
it is necessary to assume that $\hf_b$ is divisible by $\hg_b$
throughout a neighbourhood $W$ of $0$ in order to guarantee that $f = h\cdot g$, where
$h \in \cC^\infty(W)$ (quasianalytic version of the {\L}ojasiewicz-Malgrange division theorem
\cite[Thm.\,6.4]{BMselecta}). The proof in \cite{BMselecta} nevertheless works for Proposition \ref{prop:princ}, 
using the division axiom \ref{def:quasian}(1) and Corollary \ref{cor:quasiancontin}:
\end{remarks}

\begin{proof}[Proof of Proposition \ref{prop:princ}]
(Shrinking $V$ if necessary) there is a mapping $\s: M \to V$ given by a finite composite
of blowings-up as in Remarks \ref{rem:maps}(1)(a), such that $g\circ \s$ is a monomial times an invertible
factor (in suitable quasianalytic coordinates) in some neighbourhood of any point of $\s^{-1}(0)$. 
By axiom \ref{def:quasian}(1), $f\circ\s = \vp\cdot g\circ\s$, where $\vp \in \cQ(M)$. Clearly, $\hat{\vp}_a = \hs^*_a(H)$, 
for all $a \in \s^{-1}(0)$, so the result follows from Corollary \ref{cor:quasiancontin}.
\end{proof}

\begin{proposition}[$k$'th root of a quasianalytic function]\label{prop:root}
Let $\cQ$ be a quasianalytic class and let $g \in \cQ(V)$, where $V$ is a neighbourhood of 
$0$ in $\IR^n$. Then there is a quasianalytic class $\cQ' \supseteqq \cQ$ such that, if
$k$ is a positive integer and $g$ has a $k$'th root in formal power series at $0$; i.e., $\hg_0 = H^k$, where 
$H \in \cF_0$, then there is a neighbourhood $W$ of $0$ in $V$ and a quasianalytic function
$h \in \cQ'(W)$ such that $g = h^k$.
\end{proposition}

\begin{proof}
(Shrinking $V$ if necessary) there is a mapping $\s: M \to V$ given by a finite composite
of admissible blowings-up, such that $g\circ \s$ is a monomial times an invertible
factor (in suitable quasianalytic coordinates) in some neighbourhood $U$ of any point of $\s^{-1}(0)$.
By the hypothesis, this monomial is a $k$'th power, and we can take $f_U \in \cQ(U)$ such that
$g\circ\s|_U = f_U^k$ and $\hf_{U,a} = \hs^*_a(H)$, for all $a \in \s^{-1}(0)\cap U$. The result
follows from Corollary \ref{cor:quasiancontin}. 
\end{proof}

\section{Polynomial equations with quasianalytic coefficients}\label{sec:poly}

\begin{proof}[Proof of Theorem \ref{thm:poly}]
Let $\cQ(U,\IC)$ denote the ring of $\IC$-valued functions of quasianalytic class $\cQ$ on $U$.
We can consider $G(x,y)$ as an element of $\cQ(U,\IC)[y]$ and each $B_{ji} \in \IC\llb x\rrb[y]$, 
and it is enough to prove the result in the ring of polynomials with complex-valued quasianalytic coefficients.
We break the proof into a number of lemmas.

\begin{lemma}\label{lem:poly1}
We can assume that $a_1 = 0$, and that there exists $\al \in \IN^n\backslash\{0\}$ such that
$$
a_i(x)^{d!/i} = x^\al a_i^*(x),\quad i =2,\ldots, d,
$$
where each $a_i^* \in \cQ(U,\IC)$, and $a_i^*$ is a unit, for some $i$.
\end{lemma}

\begin{proof}
We can reduce to the case that $a_1 = 0$, by a coordinate change $y' = y + a_1(x)/d$. Let
$\cI$ denote the ideal sheaf generated by the functions $a_i^{d!/i}$, $i=2,\ldots,d$. The theorem
is trivial if $\cI = (0)$. Otherwise, by resolution of singularities of $\cI$, 
there is a finite composite of admissible blowings-up $\s: M \to U$
(after shrinking $U$ to a relatively compact neighbourhood of $0$) such that any point
$a \in \s^{-1}(0)$ admits a coordinate neighbourhood $W$ (with coordinates $z = (z_1,\ldots,z_n)$, say)
in which the pullback of $\cI$ is generated by a monomial $z^\al$, $\al \in \IN^{n}$; i.e.,
$$
a_i(\s(z))^{d!/i} = z^\al a^*_i(z),\quad i=2,\ldots,d,
$$
where $a_i^*$ is a unit in $\cQ(W,\IC)$, for some $i$ \cite[Thm.\,5.9]{BMselecta}. 

By Corollary \ref{cor:quasiancontin}, it is enough to find quasianalytic functions $c_{ji}(z)$
such that $\hc_{ji,a} = \hs^*_a(B_{ji})$, $j=1,\ldots,k$, $i=1,\ldots,d_j$, and
$$
G(\s(z),y) = \prod_{j=1}^k \left(y^{d_j} + c_{j1}(x)y^{d_j -1} + \cdots + c_{j,d_j}(x)\right)
$$
in $\cQ'(V,\IC)[y]$. Therefore, we can replace $G$ by $G(\s(z),y)$
and each $B_{ji}$ by $\hs^*_a(B_{ji})$ to get the lemma.
\end{proof}

\begin{notation}\label{notn:power}
For any positive integer $k$, we will write $x^k$ to denote $(x_1^k,\ldots,x_n^k)$.
For any $\ep \in \{-1,1\}^n$, we will write $\tau_\ep^k$ to denote the substitution
$$
\tau_\ep^k(x) := (\ep_1 x_1^k, \ldots, \ep_n x_n^k).
$$
\end{notation}

\begin{lemma}\label{lem:poly2}
We can assume, moreover, that
$$
a_i(x) = x^{i\al}\ta_i(x),\quad i =2,\ldots, d,
$$
where each $\ta_i \in \cQ(U,\IC)$, and $\ta_i$ is a unit, for some $i$.
\end{lemma}

\begin{proof}
For each $i=2,\ldots,d$, since $a_i(x)^{d!}$ is divisible by $x^{i\al}$, it follows that 
$a_i(x^{d!})^{d!}$ is divisible by $x^{id!\al}$, and therefore that $a_i(x^{d!})$ is divisible
by $x^{i\al}$ as a function of class $\cQ$ (using unique factorization of formal power 
series and axioms \ref{def:quasian}(1),\,(3)); i.e., 
$a_i(x^{d!}) = x^{i\al}\ta_i(x)$, $i=2,\ldots,d$, where
each $\ta_i$ is of class $\cQ$ (and $\ta_i$ is a unit, for some $i$). Likewise for
$a_i(\tau_\ep^{d!}(x))$, for any $\ep \in \{-1,1\}^n$. The assertion now follows from
Corollary \ref{cor:quasiancontin} since, according to the latter, it is enough to prove the
theorem after a power substitution $\tau_\ep^{d!}(x)$.
\end{proof}

\begin{lemma}\label{lem:poly3}
Under the assumptions of Lemmas \ref{lem:poly1}, \ref{lem:poly2}, $G(x,y)$ has a nontrivial
factorization $G = G_1G_2$ in $\cQ(U,\IC)[y]$, after perhaps shrinking $U$, and each formal
factor
$$
H_j(x,y) = y^{d_j} + B_{j1}(x)y^{d_j -1} + \cdots + B_{j,d_j}(x)
$$
splits as $H_j = H_{j1}H_{j2}$, where $H_{jl}$ is a formal factor of $G_l$, $l=1,2$ (perhaps
$H_{j1}$ or $H_{j2} =$ constant).
\end{lemma}

\begin{proof}
Since $\ta_i(0) \neq 0$, for some $i$, we can write 
$$
y^d + \ta_2(0)y^{d-2} + \cdots + \ta_d(0) \in \IC[y]
$$
as a nontrivial product of two polynomials with no common factor. Therefore, there 
is also a nontrivial  splitting
$$
y^d + \sum_{i=2}^d \ta_i(x)y^{d-i} = \left(y^k + \sum_{i=1}^k \txi_i(x)y^{k-i}\right)\cdot
                                                          \left(y^l + \sum_{i=1}^l \teta_i(x)y^{l-i}\right),
$$ 
where $k+l=d$ (see \cite[Lemma 3.1]{BMarc} and Lemma \ref{lem:resultant} below), so that
$$
G(x,y) = \left(y^k + \sum_{i=1}^k \xi_i(x)y^{k-i}\right)\cdot \left(y^l + \sum_{i=1}^l \eta_i(x)y^{l-i}\right),
$$
where $\xi_i(x) = x^{i\al}\txi_i(x)$, $i=1,\ldots,k$, and $\eta_i(x) = x^{i\al}\teta_i(x)$, $i=1,\ldots,l$.
The corresponding factorization of each $H_j$ follows from unique factorization of formal power series.
\end{proof}

Theorem \ref{thm:poly} follows, by induction on the degree $d$ of $G$.
\end{proof}

We recall \cite[Lemma 3.1]{BMarc} and its proof, since this result is needed also in Section \ref{sec:Weier}
below.

\begin{lemma}\label{lem:resultant}
Let $P(x,y) = y^d + \sum_{i=1}^d A_i(x)y^{d-i}$, where the coefficients are functions in some class
(e.g., formal power series in $x=(x_1,\ldots,x_n)$, or $\cC^\infty$
functions or functions of a quasianalytic class $\cQ$ in a neighbourhood of $0\in \IR^n$). Suppose that
$$
P(0,y) =  y^d + \sum_{i=1}^d A_i(0)y^{d-i} = Q(\be_0,y)R(\ga_0,y),
$$
where 
$$
Q(\be_0,y) = y^k + \sum_{i=1}^k \be_{0,i}y^{k-i} ,\quad R(\ga_0,y) = y^l + \sum_{i=1}^l \ga_{0,i}y^{l-i}
$$
are polynomials in $y$ with no common factor, $k+l=d$. Then
$$
P(x,y) = \left(y^k + \sum_{i=1}^k B_i(x)y^{k-i}\right)\cdot \left(y^l + \sum_{i=1}^l C_i(x)y^{l-i}\right),
$$
where the coefficients $B_i$, $C_j$ are functions of the given class.
\end{lemma}

\begin{proof} Let
$$
Q(\be,y) = y^k + \sum_{i=1}^k \be_{i}y^{k-i} ,\quad R(\ga,y) = y^l + \sum_{i=1}^l \ga_{i}y^{l-i},
$$
where $\be=(\be_1,\ldots,\be_k) \in \IR^k$, $\ga=(\ga_1,\ldots,\ga_l)\in \IR^l$. Write
$$
Q(\be,y)R(\ga,y) = y^d + \sum_{i=1}^d \al_i(\be,\ga)y^{d-i}.
$$
Then the Jaobian determinant $\D(\be,\ga) := \det \p \al(\be,\ga) / \p (\be,\ga)$ is the resultant
of $Q,\,R$ as polynomials in $y$. By the inverse function theorem, since $\D(\be_0,\ga_0)\neq 0$,
we can write
$$
y^d + \sum_{i=1}^d \al_i y^{d-i} = Q(\be(\al),y)R(\ga(\al),y),
$$
where $\be(A(0)) = \be_0$, $\ga(A(0)) = \ga_0$, $A(x) = (A_1(x),\ldots,A_d(x))$. Then the assertion 
of the lemma holds with
$B_i(x) = \be_i(A(x))$, $C_j(x) = \ga_j(A(x))$. 
\end{proof}

\section{Quasianalytic equations}\label{sec:main}

In this section, we will prove Theorem \ref{thm:main} using Corollary \ref{cor:quasiancontin}. The latter 
allows us to follow the scheme of \cite{BMV}, in a simpler way. 

\begin{lemma}\label{lem:yreg}
We can assume, without loss of generality, that, for some positive integer $d$, $G(x,y)$ is
\emph{$y$-regular} of \emph{order} $d$ at $(a,b)$; i.e., $(\p^jG/\p y^j)(a,b) = 0$ if $j < d$,
but $(\p^dG/\p y^d)(a,b) \neq 0$.
\end{lemma}

\begin{proof}
We can assume that $(a,b)=(0,0)$, so that $G(0,0)=0$ and $H(0)=0$.
Let $\vp_i := (\p^iG/\p y^i)(x,0)$, $i\in \IN$. By resolution of singularities \cite[Thm.\,3.1]{BMV},
after shrinking $U$ to a relatively compact neighbourhood of $0$, there is a $\cQ$-mapping
$\s: M \to U$ given by a finite composite of admissible blowings-up, such that any $a' \in \s^{-1}(0)$
admits a coordinate neighbourhood $W$ in which the ideal $\cJ$ generated by the restrictions of
$\vp_i\circ\s$, $i\in \IN$, is a principal ideal generated by a monomial $z^\al$, $\al\in \IN^n$,
where $z=(z_1,\ldots,z_n)$ (and $a'$ is the origin of the coordinate chart).

We claim that $G(\s(z),y)$ is divisible by $z^\al$; i.e., that $\tG(z,y) := z^{-\al}G(\s(z),y)$ is
a quasianalytic function. It is enough to show that, for each $i=1,\ldots,n$ such that $\al_i\neq 0$,
$G(\s(z),y)$ is divisible by $z_i$, or (according to axiom \ref{def:quasian}(1)), that $G(\s(z),y)$ vanishes on the
hyperplane $(z_i=0)$. For fixed $z$ such that $z_i=0$, $\Ga(y) := G(\s(z),y)$ is of class $\cQ$
and
$$
\frac{d^j\Ga}{dy^j}(0) = \frac{\p^j G}{\p y^j}(\s(z),0) = 0,\quad j\in\IN.
$$
By axiom \ref{def:quasian}(3), $\Ga$ vanishes identically, as required. Thus $\tG(z,y)$ is of class $\cQ$,
and $\tG(z,\hs^*_{a'}(H)(z)) = 0$.

Since the ideal $\cJ$ is generated by $z^\al$, there exists $d$ such that $\vp_d\circ\s|_W = z^\al$
times an invertible factor. Thus, $(\p^d\tG/\p y^d)(a',0) \neq 0$. It follows from Corollary \ref{cor:quasiancontin}
that we can assume $G(x,H(x))=0$, where $G$ is $y$-regular of some order $d$ at $0$.
\end{proof}

We now prove the theorem by induction on $d$. The case $d=1$ is a consequence of the
implicit function theorem (axiom \ref{def:quasian}(2)).
We can assume that $(a,b)=(0,0)$, so that $G(0,0)=0$ and $H(0)=0$.

\begin{lemma}\label{lem:tschirn}
We can assume that $\displaystyle{\frac{\p^{d-1}G}{\p y^{d-1}}(x,0) = 0}$.
\end{lemma}

\begin{proof}
Since $G$ is $y$-regular of order $d$ at $(0,0)$, the function $(\p^{d-1}G / \p y^{d-1})(x,y)$ has nonvanishing
derivative with respect to $y$ at $(0,0)$. By the implicit function theorem (axiom \ref{def:quasian}(2)), 
there is a function $\vp(x)$
of class $\cQ$ at $0$ such that $\vp(0)=0$ and $(\p^{d-1}G / \p y^{d-1})(x,\vp(x)) = 0$. We can replace
$G(x,y)$ by $G(x,y+\vp(x))$ and $H(x)$ by $H(x) - \hvp_0(x)$ to get the lemma.
\end{proof} 

Now, set
$$
c_i(x):= \frac{1}{(d-i)!}\cdot \frac{\p^{d-i} G}{\p y^{d-i}} (x,0),\quad i=2,\ldots,d;
$$
thus $c_1 = 0$. Taking the Taylor expansion
of $G(x,y)$ with respect to $y$, we can write
\begin{equation}\label{eq:Taylor}
G(x,y) = \rho(x,y)y^d +  \sum_{i=2}^dc_i(x)y^{d-i} \,,
\end{equation}
where $\rho$ is $\cC^\infty$ and thus of class $\cQ$ (by axiom \ref{def:quasian}(1)), and $\rho(0,0)\neq 0$.

\begin{lemma}\label{lem:res}
We can assume there exists $\al \in \IN^{n}{\setminus}\{0\}$ such that
$$
c_i(x)^{d!/i} = x^\al c^*_i(x),\quad i=2,\ldots,d,
$$
where each $c^*_i$ is of class $\cQ$ and $c_i^*$ is a unit, for some $i$.
\end{lemma}

\begin{proof}
By \eqref{eq:Taylor},
\begin{equation}\label{eq:Taylor2}
\sum_{i=0}^{d-2}c_i(x)H(x)^i + \rho(x,H(x))H(x)^d = 0,
\end{equation}
as a formal expansion at $0$. Let $\cI$ denote the ideal sheaf generated by the
functions $c_i^{d!/i}$, $i=2,\ldots,d$. If $\cI = (0)$, then $H=0$, by \eqref{eq:Taylor2},
so of course we can take $h=0$ to solve our problem. Otherwise, we apply resolution of
singularities to $\cI$, to obtain a finite composite of admissible blowings-up $\s: M \to U$
(after shrinking $U$ to a relatively compact neighbourhood of $0$) such that any point
$a \in \s^{-1}(0)$ admits a coordinate neighbourhood $W$ (with coordinates $z = (z_1,\ldots,z_n)$, say)
in which the pullback of $\cI$ is generated by a monomial $z^\al$, $\al \in \IN \setminus\{0\}$; i.e.,
$$
c_i(\s(z))^{d!/i} = z^\al c^*_i(z),\quad i=2,\ldots,d,
$$
where $c_i^*$ is a unit in $\cQ(W)$, for some $i$.

By Corollary \ref{cor:quasiancontin}, it is enough to find a quasianalytic function $h(z)$ such
that $G(\s(z),h(z))=0$ and $\hh_a=\hs^*_a(H)$. Therefore, we can replace $G$ by $G(\s(z),y)$
and $H$ by $\hs^*_a(H)$ to get the lemma.
\end{proof}

As in the proof of Lemma \ref{lem:poly2}, 
we can write $c_i(x^{d!}) = x^{i\al}\tc_i(x)$, $i=2,\ldots,d$, where
each $\tc_i$ is of class $\cQ$ (and $\tc_i$ is a unit, for some $i$). (Recall Notation \ref{notn:power}.)

Consider
\begin{align*}
G_1(x,y) &:= x^{-d\al}G(x^{d!},x^{\al}y)\\
          &{\phantom :}= \rho(x^{d!},x^\al y)y^d + \sum_{i=2}^{d}\tc_i(x)y^{d-i},\\
               H_1(x) &:= x^{-\al}H(x^{d!}).
\end{align*}
Clearly, $G_1(x,y)$ is a well-defined function of class $\cQ$ in a neighbourhood
of the $y$-axis. Since $c_{d-1} = 0$, $G_1(x,y)$ is $y$-regular of order $\leq d-1$ at any point
$(0,y_0)$. On the other hand, $H_1(x)$ is \emph{a priori} a Laurent series (with finitely many negative 
exponents). We have
$$
G_1(x,H_1(x)) = x^{-d\al}G(x^{d!},H(x^{d!})) = 0.
$$

Write
\begin{equation}\label{eq:H_1}
H_1(x) = \sum \xi_\be x_1^{\be_1}\cdots x_n^{\be_n},
\end{equation}
where the exponents $\be_j$ \emph{a priori} may be negative.

\begin{lemma}\label{lem:fract}
$H_1(x)$ is a formal power series; i.e., $H_1(x)$ has only nonnegative exponents $\be_j$.
\end{lemma}

\begin{proof}
We first check that, for any formal curve $x(t) = (x_1(t),\ldots,x_n(t))$, $x(0)=0$, 
the formal expansion $H_1(x(t))$ has nonnegative order;
i.e., $\order H(x(t)^{d!}) \geq \order x(t)^{\al}$. 

Write $K(x) := H(x^{d!})$ to simplify
the notation. Since $G_1(x,x^{-\al}K(x))=0$,
\begin{equation}\label{eq:fract}
\rho(x(t)^{d!},K(x(t)))K(x(t))^d + \sum_{i=2}^{d}\tc_i(x(t))x(t)^{i\al}K(x(t))^{d-i} = 0.
\end{equation}
Suppose that $\order K(x(t)) < \order x(t)^{\al}$. Then, for each $i$, 
$\order K(x(t))^i < \order x(t)^{i\al}$, so that 
$\order K(x(t))^d < \order x(t)^{i\al}K(x(t))^{d-i}$, in contradiction to \eqref{eq:fract}.

Now suppose there is a negative exponent $\be_j$ in \eqref{eq:H_1} (for some nonzero $\xi_\be)$.
Let $b$ denote the smallest negative exponent that occurs; we can assume that $b = \be_1$, for
some $\be = (\be_1,\ldots,\be_n)$. Let $a$ denote the smallest $\be_1 > b$ that occurs in \eqref{eq:H_1},
$A$ the smallest $\be_2 + \cdots + \be_n$ that occurs, and $B$ the smallest $\be_2 + \cdots + \be_n$
that occurs among those exponents with $\be_1 = b$.

Choose $q \in \IN$ such that $qb+B<0$ and $qb+B<qa+A$. Let $I := \{\be: \be_1 = b,\, \be_2 +\cdots + \be_n = B\}$.
Take $x(t) = (\la_1 t^q,\la_2 t,\ldots,\la_n t)$, where $\la$ is chosen so that $\sum_{\be \in I} \xi_\be \la^\be \neq 0$
($\la$ exists because $\sum_{\be \in I} \xi_\be s^\be$ is a nonzero polynomial). Then $\order H_1(x(t)) < 0$;
a contradiction.
\end{proof}

In the same way as above, for all $\ep \in \{-1,1\}^n$, define
\begin{align*}
G^\ep_1(x,y) &:= x^{-d\al}G(\tau^{d!}_\ep (x),x^{\al}y),\\
          H^\ep_1(x) &:= x^{-\al}H(\tau^{d!}_\ep (x))
\end{align*}
(cf. Notation \ref{notn:power});
then $H^\ep_1(x)$ is a formal power series, and $G^\ep_1(x,y)$ is a well-defined function of class $\cQ$ 
in a neighbourhood of the $y$-axis, which is $y$-regular of order $\leq d-1$ at any point $(0,y_0)$.

By Corollary \ref{cor:quasiancontin}, it is enough to show that there is a quasianalytic class $\cQ' \supseteqq \cQ$
with the property that, for all $\ep \in \{-1,1\}^n$, we can find a function $h^\ep$ quasianalytic of class $\cQ'$,
such that $\hh^\ep_0 = H(\tau^{d!}_\ep (x))$.

By induction on $d$, there exists $\cQ'$ with the property that, for each $\ep$, we can find $h_1^\ep$  of class
$\cQ'$ such that $G_1(\tau_\ep^{d!}(x),h_1^\ep(x))=0$ and 
$(h_1^\ep)_0^\wedge = H_1(\tau_\ep^{d!}(x))$; then we can take $h^\ep(x) := x^\al h_1^\ep(x)$.

This completes the proof of Theorem \ref{thm:main}. \qed

\section{Remarks on Weierstrass preparation}\label{sec:Weier}

Let $\cQ$ denote any subclass of $\cC^\infty$ functions which is closed under differentiation and taking 
the reciprocal of a nonvanishing function (we do not assume the axioms of Definition \ref{def:quasian},
to begin with).
A \emph{Weierstrass  polynomial} in $y$ of degree $d$ at $(0,0) \in \IR^n\times \IR$ means a function
$$
p(x,y) = y^d + a_1(x)y^{d-1} + \cdots a_d(x),
$$
where the coefficents $a_i(x) = a_i(x_1,\ldots,x_n)$ are of class $\cQ$ and vanish at $0$.

\begin{definitions}\label{def:Weier}
(1) $\cQ$ has the \emph{Weierstrass preparation property} if, for every function $g(x,y)$
of class $\cQ$ that is $y$-regular of order $d$ at $(0,0)$ (see Lemma \ref{lem:yreg}), there exists a 
Weierstrass polynomial $p(x,y)$ of degree $d$ at $(0,0)$, such that $g(x,y)=u(x,y)p(x,y)$ in a neighbourhood 
of $(0,0)$, where $u$ is a unit of class $\cQ$.

\smallskip\noindent
(2) $\cQ$ has the \emph{Weierstrass division property} if, given $f(x,y),\,g(x,y)$ of class $\cQ$, 
where $g$ is $y$-regular of order $d$ at $(0,0)$, 
$$
f(x,y) = q(x,y)g(x,y) + \sum_{i=1}^d r_i(x)y^{d-i},
$$
where $q$ and the $r_i$ are of class $\cQ$.

\smallskip\noindent
(3) $\cQ$ has the \emph{property of division by a Weierstrass polynomial} if (2) holds in the
special case that $g$ is a Weierstrass polynomial in $y$ of degree $d$.
\end{definitions}

The following lemma is classical, though it seems not so well known (see \cite[\S\,2]{CL}, \cite[Kap.\,I,\,\S4.\,Supp.\,3]{GR}).

\begin{lemma}\label{lem:Weier}
The Weierstrass preparation and division properties are equivalent. If $\cQ$ satisfies the
implicit function property (axiom \ref{def:quasian}(2)), then all three properties of Definitions \ref{def:Weier}
are equivalent.
\end{lemma}

\begin{proof} 
We first show that Weierstrass preparation in $k+1$ variables implies Weierstrass division in $k$ variables.
Suppose that $g(x,y) = g(x_1,\ldots,x_n,y)$ is of class $\cQ$ and $y$-regular of order $d$ at $(0,0)$. Let 
$f(x,y)$ be a function of class $\cQ$. We want to divide $f$ by $g$. By Weierstrass preparation, $g(x,y) = u(x,y)p(x,y)$
in class $\cQ$, where $u$ is a unit and $p$ is a Weierstrass polynomial $p(x,y) = y^d + \sum_{i=1}^d a_i(x)y^{d-i}$.

The function $F(x,y,t) = p(x,y) + tf(x,y)$ is $y$-regular of order $d$ at $(0,0,0)$. By Weierstrass preparation,
\begin{equation}\label{eq:Weier1}
p(x,y) + tf(x,y) = U(x,y,t)P(x,y,t),
\end{equation}
where $U$ is a unit and $P$ is a Weierstrass polynomial in $y$ of degree $d$.
Clearly, $U(x,y,0) = 1$ and $P(x,y,0) = p(x,y)$. Let
$$
r(x,y) = \left.\frac{\p P(x,y,t)}{\p t}\right|_{t=0};
$$
then $r(x,y)$ is a polynomial of degree $<d$ in $y$. Apply $\p /\p t$ to \eqref{eq:Weier1} and set $t=0$; we
get
\begin{equation*}
f(x,y) = q(x,y)g(x,y) + r(x,y),
\end{equation*}
where
$$
q(x,y) = \left.\frac{\p U(x,y,t)}{\p t}\right|_{t=0}\cdot u(x,y)^{-1},
$$
as required.

Clearly, Weierstrass division in $k$ variables implies Weierstrass preparation in $k$ variables.
(Given $g(x,y)$ regular of order $d$ in $y$, divide $y^d$ by $g$ and subtract the remainder term.)

Now assume that $\cQ$ satisfies the implicit function property and the property of division by a
Weierstrass polynomial. Let $P(\la,y)$ denote the \emph{generic
polynomial} of degree $d$,
$$
P(\la,y) := y^d + \sum_{i=1}^d \la_i y^{d-i}.
$$
Given $g(x,y)$, divide by $P(\la,y)$ as functions of $(x,\la,y)$:
\begin{equation}\label{eq:Weier2}
g(x,y) = q(x,\la,y)P(\la,y) + \sum_{i=1}^d r_i(x,\la) y^{d-i}.
\end{equation}

Suppose that $g(x,y)$ is $y$-regular of order $d$ at $(0,0)$. Put $x = 0 =\la$ in \eqref{eq:Weier2};
then $\text{unit}\cdot y^d = q(0,0,y)y^d + \sum r_i(0,0)y^{d-i}$. Clearly, $q(0,0,0)\neq 0$ and $r_i(0,0)=0$,
for all $i$. It is easy to check that the Jacobian determinant $\det (\p r_i /\p \la_j)$ does not vanish at $(0,0)$,
so that the system of equations $r_i(x,\la)=0$ has a solution $\la = \vp(x)$, $\vp(0)=0$, and we get Weierstrass
preparation $g(x,y) = q(x,\vp(x),y)P(\vp(x),y)$.
\end{proof}

\begin{definition}\label{def:hyper}
Let $g(x,y)=g(x_1,\ldots,x_n,y)$ denote a $\cC^\infty$ function that is $y$-regular of order $d$
at $(0,0)\in \IR^n\times \IR$. We say that $g$ is $y$-hyperbolic at $(0,0)$ if, for every $\ep>0$, there
exists $\de>0$ such that, for any given $x$ such that $|x|<\de$ (where $|x| := (x_1^2 +\cdots + x_n^2)^{1/2}$),
$g(x,y)=0$ has $d$ real roots (counted with multiplicity) in the interval $(-\ep,\ep)$.
\end{definition}

\begin{theorem}\label{thm:Weier}
Let $g(x,y)$ be a function of quasianalytic class $\cQ$ in a neighbourhood of $(0,0) \in \IR^n\times \IR$. 
Assume that $g$ is regular of order $d$ and hyperbolic with respect to $y$ at $(0,0)$. Then there is a
(perhaps larger) quasianalytic class $\cQ' \supseteq \cQ$ such that $g(x,y)=u(x,y)p(x,y)$ near
$(0,0)$, where 
$$
p(x,y) = y^d + a_1(x)y^{d-1} + \cdots a_d(x)
$$
is a Weierstrass polynomial with coefficients $a_i(x)$ of class $\cQ'$, and $u$ is a unit of class $\cQ'$.
\end{theorem}

\begin{proof}
The proof is by induction on $d$. By the formal Weierstrass preparation theorem,
$$
\hg_{(0,0)}(x,y) = U(x,y)H(x,y),
$$
where 
$$
H(x,y) = y^d + \sum_{i=1}^d B_i(x) y^{d-i} \in \IR\llb x\rrb[y]
$$
and $U(x,y)$ is a unit in $\IR\llb x,y\rrb$.

We argue as in the proof of Theorem \ref{thm:main} (Section \ref{sec:main}). Note that the
hyperbolicity property of $g$ is preserved after pull-back to any point in the inverse image of the 
origin, by a blowing-up in the $x$-variables. We can assume that $(\p^{d-1} g / \p y^{d-1})(x,0)=0$
(Lemma \ref{lem:tschirn}), so that
$$
g(x,y) = \rho(x,y)y^d +  \sum_{i=2}^dc_i(x)y^{d-i} \,,
$$
as in \eqref{eq:Taylor}, where $\rho(0,0)\neq 0$. We can also assume (as in Lemma \ref{lem:res} ff.) 
that there exists $\al \in \IN^{n}{\backslash}\{0\}$ such that 
$c_i(x^{d!}) = x^{i\al} \tc_i(x)$, $i=2,\ldots, d$, where each $\tc_i$ is of class $\cQ$ and $\tc_i$ is 
a unit, for some $i$.

Consider
\begin{align*}
G_1(x,y) &:= x^{-d\al}g(x^{d!},x^{\al}y) = \rho(x^{d!},x^\al y)y^d + \sum_{i=2}^{d}\tc_i(x)y^{d-i},\\
H_1(x,y) &:= x^{-d\al}H(x^{d!},x^\al y) = y^d + \sum_{i=1}^{d}x^{-i\al}B_i(x^{d!})y^{d-i}\,.
\end{align*}
Then $G_1(x,y) = U(x^{d!},x^\al y) H_1(x,y)$ as formal expansions. Setting $y=0$, we see that
$\tc_d(x) = U(x^{d!},0)x^{-d\al}B_d(x^{d!})$, so that $B_d(x^{d!})$ is divisible by $x^{-d\al}$; i.e., 
$x^{-d\al}B_d(x^{d!})$ is a formal power series
$\tB_d(x)$. Successively taking $\p^j /\p y^j$, $j=1,2,\ldots$ and setting $y=0$, we see that each
$x^{-i\al}B_i(x^{d!})$ is a formal power series $\tB_i(x)$.

Setting $x=0$, we have
$$
\rho(0,0)y^d + \sum_{i=2}^{d}\tc_i(0)y^{d-i} = U(0,0) \left(y^d + \sum_{i=1}^{d}\tB_i(0)y^{d-i}\right).
$$
Therefore, $\rho(0,0)=U(0,0)$, $\tB_1(0)=0$, and $\tc_i(0) = U(0,0)\tB_i(0)$, $i=2,\ldots,d$.

We claim that all roots of $y^d + \sum_{i=1}^{d}\tB_i(0)y^{d-i} = 0$
are real. In fact, by the Malgrange preparation theorem \cite[Ch.\,V]{Malg}, we can write
$g(x,y) = u(x,y)h(x,y)$ in a neighbourhood of the origin, where $u(x,y)$
is a nonvanishing $\cC^\infty$ function, and
$$
h(x,y) = y^d + \sum_{i=1}^{d} b_i(x)y^{d-i},
$$
where, for each $i$, $b_i(x)$ is $\cC^\infty$ and $b_i(0)=0$. It follows
that $h(x,y)$ is $y$-hyperbolic at $0$, and $B_i(x)$ is the formal Taylor
expansion at $0$ of $b_i(x)$, for each $i$. Set $\tb_i(x) := b_i(x^{d!})/x^{i\al}$, $i=1,\ldots,d$.
Then each $\tb_i(x)$ is a $\cC^\infty$ function in a neighbourhood of $0$ (as can
be seen by taking successive derivatives with respect to $y$ of the equation
$G_1(x,y) = u(x^{d!}, x^{\al}y) h_1(x,y)$, where $h_1(x,y) := y^d + \sum_{i=1}^{d} \tb_i(x)y^{d-i}$, 
and then setting $y=0$). Moreover, $\tb_i(0) = \tB_i(0)$, for each $i$, and
$y^d + \sum_{i=1}^{d}\tB_i(0)y^{d-i} = 0$
has $d$ complex roots. So these are all real, by continuity of the roots of $h_1(x,y)=0$
(for example, on a line $(x_1,\ldots,x_n)=(t,t,\ldots,t)$).

Moreover, since $\tB_1(0)=0$ and $\tB_i(0)$ is a unit, for some $i$, we can write
$$
y^d + \sum_{i=1}^{d}\tB_i(0)y^{d-i} = \prod_{j=1}^q \left(y-\la_j\right)^{d_j},
$$
where the $\la_j$ are distinct real numbers and $q\geq 2$; thus each $d_j < d$.
It follows from Lemma \ref{lem:resultant} that
$$
H_1(x,y) = \prod_{j=1}^q \left(y^{d_j} + \sum_{i=1}^{d_j} B_{ji}(x)y^{d_j -i}\right),
$$
where all $B_{ji}(x) \in \IR\llb x\rrb$.

For each $j$, $G(x, \la_j + y)$ is $y$-regular of order $d_j$ and hyperbolic at $(0,0)$.
By induction, there is a quasianalytic class $\cQ' \supseteq \cQ$, such that every coefficient
$B_{ji}(x)$ is the formal Taylor expansion at $0$ of a function $b_{ji}(x)$ of class $\cQ'$.

The argument above applies equally to $G_1^\ep(x,y)$ and $H_1^\ep(x,y)$ (as defined in
the proof of Theorem \ref{thm:main}), for any $\ep \in \{-1,1\}^n$, so the conclusion of 
Theorem \ref{thm:Weier} follows also as in the proof of Theorem \ref{thm:main}. (Of course,
$p(x,y)$ coincides with the function $h(x,y)$ above.)
\end{proof}

\begin{remarks}\label{rem:Weier}
(1)\, Chaumat and Chollet proved that division by a hyperbolic Weierstrass polynomial of quasianalytic
class $\cQ$ (i.e., division according to property (3) of Definitions \ref{def:Weier}) holds with no loss of regularity
in the quotient and remainder \cite{CC}. It follows from \cite{CC} together with
Theorem \ref{thm:Weier} that, if $g(x,y)$ is a function of class $\cQ$ that is regular and hyperbolic with respect
to $y$, then Weierstrass division by $g(x,y)$ (property (2) above) holds with loss of regularity given by
Theorem \ref{thm:Weier}. It is not evident that this result follows directly either from \cite{CC} or from
Theorem \ref{thm:Weier} (compare with the implications (1)$\implies$(2) and (3)$\implies$(1) in the proof
of Lemma \ref{lem:Weier} above).

\smallskip\noindent
(2)\, Let $\cQ$ be a quasianalytic Denjoy-Carleman class $\cQ_M$, and let $g(x,y)$ denote a
hyperbolic Weierstrass polynomial of class $\cQ_M$, of degree $d$ in $y$ (for $x$ in a neighbourhood 
of $0$ in $\IR^n$). Let $a_0 = (0,0) \in \IR^n\times \IR$. If $f(x,y)$ is of class $\cQ_M$ and $\hf_{a_0} = H\cdot\hg_{a_0}$,
where $H \in \IR\llb x,y\rrb$, then there exists $h(x,y)$ of class $\cQ_M$ in a neighbourhood of $a_0$, such
that $f=h\cdot g$. In fact, by Proposition \ref{prop:princ}, for each $x$ near $0$, all roots of $g(x,y)=0$
are roots of $f(x,y)=0$. By \cite{CC}, $f(x,y)=q(x,y)g(x,y)+r(x,y)$ near $a_0$, where $q,\,r$ are of class $\cQ_M$,
and $r(x,y)$ is a polynomial in $y$ of degree $<d$; therefore, $r=0$.

\smallskip\noindent
(3)\, The loss of regularity in Theorem \ref{thm:Weier} depends on the function $g(x,y)$. For any given quasianalytic 
Denjoy-Carleman class $\cQ_M$ (except $\cQ_M = \cO$), it is not true that there exists a (perhaps larger) 
Denjoy-Carleman class $\cQ_{M'}$, such that, given $g(x,y)$ $y$-regular of class $\cQ_M$ (or a Weierstrass
polynomial of class $\cQ_M$), any function $f(x,y)$ of class $\cQ_M$ admits Weierstrass division by $g(x,y)$
with quotient and remainder of class $\cQ_{M'}$ \cite{ABBNZ}, \cite{L}, \cite{T}. 

It is easy to see that, given $\cQ_M$ and $f \in \cQ_M([0,1))$ (say $f(0)=0,\, f'(0)>0$), the function
$g(x,y):=f(y^2)-x$ (which is $y$-regular of order $2$) satisfies the Weierstrass preparation property
with $u,p \in \cQ'$, for some quasianalytic class $\cQ' \supseteqq \cQ_M$, if and only if 
$f$ extends to a function in $\cQ'((-\de,1))$, for some $\de > 0$. Nazarov, Sodin and Volberg \cite{NSV} showed, 
in fact, that there is a quasianalytic 
Denjoy-Carleman class $\cQ_M$, and a function $f \in \cQ_M([0,1))$ which admits no extension to a function
in $\cQ_{M'}((-\de,1))$, for any quasianalytic $\cQ_{M'}$ and $\de > 0$.

It seems interesting to ask whether a Denjoy-Carleman class $\cQ_M$ nevertheless does have the
Weierstrass division property or the extension property as above, where $\cQ'$ is some quasianalytic class that depends 
on the given functions.

\smallskip\noindent
(4)\, The point of view of this article seems relevant to the study of many algebraic properties 
of local rings of quasianalytic functions. For example, it is unknown (and unlikely) that such local rings 
are Noetherian, in general, although a topological version of Noetherianity follows from resolution 
of singularities \cite[Thm.\,6.1]{BMselecta}; cf. the proof of Corollary \ref{cor:contin} above. 
The local ring $\cQ_n$ of germs of functions of quasianalytic class $\cQ$ at the origin of $\IR^n$ is 
Noetherian if and only if, for all $f,g_1,\ldots g_p \in \cQ_n$, the equation 
$f(x) = \sum_{i=1}^p y_i g_i(x)$ has a solution $y_i = h_i(x)$, $i=1,\ldots,p$,
of class $\cQ$, provided there is a formal power series solution $y_i = H_i(x)$. 
One can ask whether this formal condition implies rather the existence of a
quasianalytic solution with loss of regularity depending on $g_1,\ldots g_p$ and perhaps $f$.
\end{remarks}

\bibliographystyle{amsplain}

\begin{thebibliography}{99}

\bibitem{ABBNZ}
F. Acquistapace, F. Broglia, M. Bronshtein, A.Nicoara and N. Zobin, 
\textit{Failure of the Weierstrass preparation theorem in quasi-analytic Denjoy-Carleman rings},
Adv. Math. \textbf{258} (2014), 397--413.

\bibitem{Bib}
I. Biborski,
\textit{On the geometric and differential properties of closed sets definable in quasianalytic
structures},
preprint, 2015, 32 pages, arXiv:1511.05071v1 [math.AG].

\bibitem{BMarc}
E. Bierstone and P.D. Milman,
\textit{Arc-analytic functions}, Invent. Math. \textbf{101} (1990), 411--424.

\bibitem{BMrel2}
E. Bierstone and P.D. Milman,
\textit{Relations among analytic functions, II},
Ann. Inst. Fourier (Grenoble) \textbf{37:2} (1987), 49--77.

\bibitem{BMinv}
E. Bierstone and P.D. Milman,
\textit{Canonical desingularization in characteristic zero by
blowing up the maximum strata of a local invariant},
Invent. Math. \textbf{128} (1997), 207--302.

\bibitem{BMselecta}
E. Bierstone and P.D. Milman,
\textit{Resolution of singularities in Denjoy-Carleman classes}, Selecta Math. (N.S.)
\textbf{10} (2004), 1--28.

\bibitem{BMV}
E. Bierstone, P.D. Milman and G. Valette,
\textit{Arc-quasianalytic functions}, Proc. Amer. Math. Soc. \textbf{143} (2015),
3915--3925.

\bibitem{Borel}
E. Borel,
\textit{Sur la g\'en\'eralisation du prolongement analytique},
C. R. Acad. Sci. Paris \textbf{130} (1900), 1115--1118.

\bibitem{Carl}
T. Carleman,
\textit{Les Fonctions Quasi-analytiques},
Collection Borel, Gauthier-Villars, Paris, 1926.

\bibitem{CC}
J. Chaumat and A.-M. Chollet, Division par un polyn\^ome hyperbolique,
\textit{Can. J. Math.} \textbf{56} (2004), 1121-1144.

\bibitem{Child}
C.L. Childress,
\textit{Weierstrass division in quasianalytic local rings},
Can. J. Math. \textbf{28} (1976), 938--953.

\bibitem{CL}
R. Cluckers and L. Lipshitz,
\textit{Strictly convergent analytic structures}, 
J. Euro. Math. Soc. \textbf{19} (2017), 107--149.

\bibitem{Den}
A. Denjoy,
\textit{Sur les fonctions quasi-analytiques de variable re\'elle},
C. R. Acad. Sci. Paris \textbf{173} (1921), 3120--1322.

\bibitem{G}
G. Glaeser, \textit{Fonctions compos\'ees diff\'erentiables}, Ann. of Math. \textbf{77} (1963),
193--209.

\bibitem{GR}
H. Grauert and R. Remmert,
\textit{Analytische Stellenalgebren}, Grundlehren der mathematischen
Wissenschaften \textbf{176}, Springer-Verlag, Berlin-Heidelberg-New York, 1971.

\bibitem{Had}
J. Hadamard,
\textit{Lectures on Cauchy's Problem in LInear Partial Differential Equations},
Yale Univ. Press, New Haven, 1923.

\bibitem{Horm}
L. H\"ormander,
\textit{The Analysis of Linear Partial Differential Operators I},
Springer-Verlag, Berlin-Heidelberg-New York, 1983.

\bibitem{Kom}
H. Komatsu,
\textit{The implicit function theorem for ultradifferentiable mappings},
Proc. Japan Acad. \textbf{55} (1979), 69--72.

\bibitem{L}
M. Langenbruch,
\textit{Extension of ultradifferentiable functions},
Manuscripta Math. \textbf{83} (1994), 123--143.

\bibitem{Malg}
B. Malgrange,
\textit{Ideals of Differentiable Functions}, 
Tata Institute of Fundamental Research Studies in Math. \textbf{3}, 
Oxford University Press, London 1966.

\bibitem{Mandel}
S. Mandelbrojt,
\textit{S\'eries Adh\'erentes, R\'egularisation des Suites, Applications},
Collection Borel, Gauthiers-Villars, Paris, 1952.

\bibitem{Mil}
C. Miller, 
\textit{Infinite differentiability in polynomially bounded $o$-minimal structures},
Proc. Amer. Math. Soc. \textbf{123} (1995), 2551--2555.

\bibitem{NSV}
F. Nazarov, M. Sodin, A. Volberg,
\textit{Lower bounds for quasianalytic functions. I. How to control smooth functions},
Math. Scand. \textbf{95} (2004), 59--79.

\bibitem{Ncomp}
K.J. Nowak, 
\textit{A note on Bierstone-Milman-Paw{\l}ucki's paper ``Composite differentiable functions''},
Ann. Polon. Math. \textbf{102} (2011), 293--299.

\bibitem{Ndiv}
K.J. Nowak, 
\textit{On division of quasianalytic function germs}, Int. J. Math. \textbf{13} (2013), 1--5.

\bibitem{Nelim} 
K.J. Nowak,
\textit{Quantifier elimination in quasianalytic structures via non-standard analysis}, 
Ann. Polon. Math. \textbf{114} (2015), 235--267.

\bibitem{RSW}
 J.-P. Rolin, P. Speissegger and A. J. Wilkie,
 \textit{Quasianalytic Denjoy-Carleman classes and o-minimality},
J. Amer. Math. Soc. \textbf{16} (2003), 751--777.

\bibitem{Rou}
C. Roumieu,
\textit{Ultradistributions d\'efinies sur $\IR^n$ et sur certaines classes de vari\'et\'es diff\'erentiables},
J. Analyse Math. \textbf{10} (1962--63), 153--192.

\bibitem{Th1}
V. Thilliez,
\textit{On quasianalytic local rings},
Expo. Math. \textbf{26} (2008), 1--23.

\bibitem{T}
V. Thilliez,
\emph{On the non-extendability of quasianalytic germs},
preprint, 2010, 4 pages, arXiv:1006:4171v1 [mathCA].

\end{thebibliography}

\end{document}